\documentclass[12pt]{amsart}
\usepackage{amssymb,latexsym}
\usepackage{enumerate}
\usepackage{upref}
\usepackage[backref,bookmarks=false]{hyperref}

\makeatletter
\@namedef{subjclassname@2010}{%
\textup{2010} Mathematics Subject Classification}
\makeatother

\newcommand{\ben}{\begin{enumerate}}
\newcommand{\een}{\end{enumerate}}

\newcommand{\bea}{\begin{eqnarray}}
\newcommand{\ba}{\begin{array}}
\newcommand{\bean}{\begin{eqnarray*}}
\newcommand{\ea}{\end{array}}
\newcommand{\eea}{\end{eqnarray}}
\newcommand{\eean}{\end{eqnarray*}}
\newcommand{\beq}{\begin{equation}}
\newcommand{\eeq}{\end{equation}}
\newcommand{\bthm}{\begin{thm}}
\newcommand{\ethm}{\end{thm}}
\newcommand{\blem}{\begin{lem}}
\newcommand{\elem}{\end{lem}}
\newcommand{\bprop}{\begin{prop}}
\newcommand{\eprop}{\end{prop}}
\newcommand{\bproperty}{\begin{property}}
\newcommand{\eproperty}{\end{property}}
\newcommand{\bcor}{\begin{cor}}
\newcommand{\ecor}{\end{cor}}
\newcommand{\bdfn}{\begin{defin}}
\newcommand{\edfn}{\end{defin}}
\newcommand{\brem}{\begin{rem}}
\newcommand{\erem}{\end{rem}}
\newcommand{\bpf}{\begin{proof}}
\newcommand{\epf}{\end{proof}}
\newcommand{\bfact}{\begin{fact}}
\newcommand{\efact}{\end{fact}}

\newcommand{\nl}{\newline}

\alph{enumii} \roman{enumiii}

\newtheorem{thm}{Theorem}[section]
\newtheorem{cor}[thm]{Corollary}
\newtheorem{lem}[thm]{Lemma}

\newtheorem{prop}[thm]{Proposition}
\newtheorem{property}[thm]{Property}

\newtheorem{fact}[thm]{Fact}

\theoremstyle{definition}
\newtheorem{defin}[thm]{Definition}
\newtheorem{rem}[thm]{Remark}

\numberwithin{equation}{section}

             \def\cB{\mathcal B}

\def\cS{\mathcal S}

\def\N{{\mathbb N}}            \def\Z{{\mathbb Z}}      \def\R{{\mathbb R}}

\def\1{1\!\!1}
\def\and{\text{ and }}

\def\dist{\text{{\rm dist}}}

\def\h{{\text h}}             
\def\hmu{\h_\mu}           
\def\H{\text{{\rm H}}}     \def\HD{\text{{\rm HD}}}   
         \def\PD{\text{\rm {PD}}}  
    
\def\Int{\text{{\rm Int}}}  

               \def\e{\varepsilon}          
\def\g{\gamma}                           \def\l{\lambda}
\def\La{\Lambda}              \def\om{\omega}           \def\Om{\Omega}
               \def\sg{\sigma}
               \def\th{\theta}

\def\bi{\bigcap}              \def\bu{\bigcup}
\def\({\bigl(}                \def\){\bigr)}
\def\lt{\left}                \def\rt{\right}

\def\ld{\ldots}                        \def\^{\tilde}

\def\es{\emptyset}            \def\sms{\setminus}
\def\sbt{\subset}             \def\spt{\supset}

           \def\downto{\searrow}
\def\sp{\medskip}             \def\fr{\noindent}        \def\nl{\newline}
\def\ov{\overline}            \def\un{\underline}
\def\vp{\varepsilon}
\def\ess{{\rm ess}}           
\def\om{\omega}

\def\J{\mathcal J}

\begin{document}


\baselineskip=17pt



\title{Random countable iterated function systems with overlaps and applications}

\author[Eugen Mihailescu]{Eugen Mihailescu}
\address{Eugen Mihailescu\nl 
Institute of Mathematics of the Romanian Academy\\
P.O Box 1-764\\
RO 014700, Bucharest, 
Romania, \ \  and\nl
Institut des Hautes \'Etudes Sci\'entifiques, Le Bois-Marie 35, route de Chartres 91440, Bures-sur-Yvette, France.
}
\email{Eugen.Mihailescu@imar.ro\newline \hspace*{0.1cm}
Web: www.imar.ro/$\sim$mihailes}

\author[Mariusz Urba\'nski]{Mariusz Urba\'nski}

\address{Mariusz Urba\'nski \nl Department of Mathematics\\University of North Texas\\
  Denton, TX 76203-1430, USA.}  
\email{urbanski@unt.edu\newline \hspace*{0.1cm}
Web: http://www.math.unt.edu/$\sim$urbanski}
\date{}

\begin{abstract}
We study invariant measures for random countable (finite or infinite) conformal iterated function systems (IFS) with arbitrary overlaps. We do not assume any type of separation condition. We prove, under a mild assumption of finite entropy, the  dimensional exactness of the projections of invariant measures from the shift space, and we give a formula for their dimension, in the context of random infinite conformal iterated function systems with overlaps. There exist many differences between our case and the finite deterministic case studied in \cite{FH}, and we introduce new methods specific to the infinite and random case. We apply our results towards a problem related to a conjecture of Lyons about random continued fractions (\cite{L}), and show that for Lebesgue-almost all parameters $\lambda>0$, the invariant measure $\nu_\lambda$ is exact dimensional. The finite IFS determining these continued fractions is not hyperbolic, but we can associate to it a random infinite IFS of contractions which have overlaps. We study then also other large classes of random countable iterated function systems with overlaps, namely: a) several types of random iterated function systems related to Kahane-Salem sets; and b) randomized infinite IFS in the plane which have  uniformly bounded number of disc overlaps. For all the above classes, we find lower and upper estimates for the pointwise (and Hausdorff, packing) dimensions of the invariant measures. 
\end{abstract}

\maketitle

\section{Introduction}

\fr Let $(X,\rho)$ be a metric space. A finite Borel measure $\mu$ on $X$ is called \textit{exact dimensional} if 
\beq\label{120140109}
d_\mu(x):=\lim_{r\to 0}\frac{\log\mu(B(x,r))}{\log r} 
\eeq
exists for $\mu$-a.e. $x\in X$ and is equal to a common value  denoted by $d_\mu$. Exact dimensionality of the measure $\mu$ has profound geometric consequences (for eg \cite{Fa}, \cite{FH},  \cite{Mattila}, \cite{Pe}, \cite{pubook}).

\fr The question of which measures are exact dimensional attracted the attention at least since    
the seminal paper of L.S Young \cite{lsy1}, where it was proved  a formula for the Hausdorff dimension
of a hyperbolic measure invariant under a surface diffeomorphism, formula involving the Lyapunov exponents of the measure. As a
consequence of that proof, she established what (now) is called
the dimensional exactness of such measures. The topic of dimensional
exactness was then pursued by the breakthrough result of Barreira,
Pesin, and Schmeling who proved in \cite{BPS} the Eckmann--Ruelle conjecture asserting 
that any hyperbolic measure invariant under smooth diffeomorphisms is exact dimensional (\cite{ER}). Dimensional exactness, without using these words, was also established in the book \cite{gdms} for all projected invariant measures 
with finite entropy, in the setting of conformal iterated function systems with
countable alphabet which  satisfy the Open Set
Condition (OSC);  in particular for all projected invariant measures if the
alphabet is finite and we have OSC. 
The next difficult task was the
case of a conformal iterated function system with \textit{overlaps}, i.e. without
assuming the Open Set Condition. For the case of iterated function systems with finite alphabet and having overlaps, this was done by Feng and Hu in \cite{FH}. 
Overlaps in iterated function systems (IFS) are challenging. Our goal in the present paper is to extend the above mentioned paper of Feng and Hu, in \textit{two directions}. Firstly, by allowing the alphabet of the system to be \textit{countable infinite}; and secondly, to consider \textit{random} iterated function systems rather than deterministic IFS. Random IFS's contain deterministic IFS as a special case. 

We  prove under a mild assumption of finite conditional entropy, the dimensional exactness of the projections of invariant measures from the shift space, in the context of random conformal iterated function systems with countable alphabet and arbitrary overlaps. We thus deal simultaneously with two 
new, and qualitatively different issues: infinite alphabet rather than finite, and random, rather
than deterministic choice of contractions. We thus need new ideas and techniques appropriate
to the context of infinite alphabet and randomness. 
Randomization allows to have a unitary setting to study limit sets and measures, in a family of systems for \textit{generic} parameter values, which proves useful in cases when studying individual systems is difficult. Moreover, randomization allows us to obtain new types of fractal sets defined by \textit{series of random variables}. 

\

Our \textbf{main results} are the following: \ in Theorem \ref{l3ie27} we prove \textit{dimensional exactness}, and provide a \textit{formula} for the dimension of typical projection measures, by employing a random projectional entropy and the Lyapunov exponents of the measure with respect to the random countable IFS with overlaps. Thus, we show that the pointwise, Hausdorff and packing dimensions of such a typical projection measure, coincide. 

Then, in Theorem \ref{ab} we give  lower and upper bounds for the random projectional entropy of a measure. This allows us consequently, to obtain estimates for the pointwise dimension (and thus Hausdorff dimension, and packing dimension) of projection measures.  

In Section 3, we introduce and investigate several \textit{concrete classes} of random countable iterated systems with overlaps. Firstly, we will give several ways to randomize countable IFS related to \textit{generalizations of Kahane-Salem sets} (\cite{KS}),  and infinite convolutions of Bernoulli distributions.  Some of these are fractal sets obtained from series of random variables, namely sets $
\lt\{\pm 1 + \mathop{\sum}\limits_{i\ge 1} \mathop{\sum}\limits_{(j, k) \in Z_i}\pm \rho_1^k\rho_2^j\rt\}
$,
where for any pair of positive integers $(j, k) \in Z_i$ we have $j+k = i, \ i \ge 1$, and where the sets $Z_i$ are prescribed by the parameter $\lambda \in \{1, 2\}^\Z$, and the signs $\pm$ are arbitrary. We obtain also another type of random fractal sets, defined  in the following way:
 $J_\lambda= \{\pm 1 \pm \lambda \rho_{i_1} \pm \lambda^2 \rho_{i_1}\rho_{i_2} \pm \ldots,  \ \text{for all sequences of positive integers} \  \omega = (i_1, i_2, \ldots) \in E^\infty\}$. We provide estimates for the pointwise dimensions of invariant measures on these random fractals.

We will then study IFS related to \textit{random continued fractions}, and will shed new light on a problem of Lyons (see \cite{L}, \cite{SSU}). These are continued fractions of type $[1, X_1, 1, X_2, \ldots] = \frac{1}{1+ \frac{1}{X_1 + \frac{1}{1+\frac{1}{X_2 +\ldots}}}}$, where the random variables $X_i, i \ge 1$ take two values $0, \alpha>0$, each with equal probability; the distribution of this random continued fraction is denoted by $\nu_\alpha$. The above random continued fractions correspond to a parabolic IFS with overlaps, whose limit set contains an interval in certain cases. In \cite{SSU}, by using a transversality condition,  it was shown that for a certain interval of parameters $\alpha$, the invariant measure $\nu_\alpha$ is absolutely continuous;  also for other parameters $\nu_\alpha$ is singular.  In our current paper, we do not use  transversality. One can associate to the above finite  parabolic IFS, a random infinite hyperbolic IFS with overlaps.  In Theorem \ref{lyons} we will prove \textit{exact dimensionality} of the measure $\nu_\alpha$, for almost all values of the parameters $\alpha$. And moreover, we will give \textit{lower estimates} of the pointwise dimension (and Hausdorff dimension, packing dimension) of $\nu_\alpha$. Our method can be extended also to other types of random continued fractions, and associated infinite IFS with overlaps.

Moreover in Section 3,  we provide examples of random infinite conformal IFS with overlaps in the plane, with \textit{uniformly bounded preimage counting function}.  We study the projection measures on their limit sets, finding lower and upper bounds for their pointwise dimensions. 

\

In general, infinite IFS with overlaps \textbf{behave differently} than finite IFS with overlaps (for eg \cite{gdms}, \cite{Mo}, etc). In the infinite case, the limit set is not necessarily compact (by contrast to the finite IFS case), also the diameters of the sets $\phi_i(X)$ converge to 0, etc. In addition, for an infinite IFS $\cS$, the boundary at infinity $\partial_\infty(\cS)$ plays an important role, and we have to take into consideration whether an invariant probability gives measure zero (or not) to $\partial_\infty(\cS)$ (for eg \cite{gdms}).  
Even when OSC is satisfied, the Hausdorff dimension of the limit set is not always given as the zero of the pressure of a certain potential. However, a version of Bowen's formula for the Hausdorff dimension still exists; see \cite{gdms}. For example even when assuming OSC, and unlike in the finite alphabet case, the Hausdorff measure can vanish and the packing measure may become locally infinite at every point. In addition for infinite systems with overlaps we may have infinitely many basic sets overlapping at points in the limit set $J$, or the number of overlaps may be unbounded over $J$. Also, in \cite{RU} there was studied the thermodynamic formalism for random IFS which satisfy the Open Set Condition. 
 In \cite{MU}, we obtained lower estimates for the Hausdorff dimension of the limit set $J$ of a deterministic infinite IFS with overlaps, by using a pressure function and a preimage counting function for the overlaps at various points of $J$.

By extension, the case of \textit{random infinite IFS with overlaps} presents even more differences and new phenomena, when compared to the case of finite IFS with overlaps. For instance several proofs that used compactness type arguments cannot be applied to random infinite IFS with overlaps. 
We also have to impose certain conditions on the randomization process $\theta: \Lambda \to \Lambda$ and on the invariant measure $\mu$ on $\Lambda \times E^\infty$, etc. Therefore, we develop several new methods in our current paper.

We mention that several authors investigated the question of dimension for measures in the context of random dynamical systems or finite iterated function systems, although in different settings than us; for example \cite{Bogenschutz},  \cite{K}, \cite{MSU},  \cite{PSS}, \cite{PW}, \cite{Pe}, \cite{SU}. 

\

\section{Pointwise dimensions for self-measures of random countable IFS with overlaps. }
First let us recall some well-known geometric concepts, see for eg \cite{Fa}, \cite{Pe}, \cite{pubook}). For a finite Borel measure $\mu$ on a metric space $(X, \rho)$, we denote by $\un d_\mu(x)$ and $\ov d_\mu(x)$ the lower and, respectively the  upper limit when $r \to 0$,  of the ratio:
$$\frac{\log\mu(B(x,r))}{\log r}$$
These limits are called respectively, the \textit{lower, and upper pointwise dimensions} of $\mu$ at $x$, and are guaranteed to exist at every $x \in X$, in contrast to the limit in \eqref{120140109}.
If $\un d_\mu(x) = \ov d_\mu(x)$, then the common value is called the \textit{pointwise dimension} of $\mu$ at $x$, and is denoted by $d_\mu(x)$. 
Now one can define the following dimensions:
$$
\HD_\star(\mu):=\inf\{\HD(Y): \mu( Y)>0\} \ \text{and} \
\HD^\star(\mu)=\inf\{\HD(Y): \mu(X\sms Y)=0\}.
$$
In the case when $\HD_\star(\mu)=\HD^\star(\mu)$, this common
value is called the \textit{Hausdorff dimension of the measure} $\mu$ and is denoted by $\HD(\mu)$.

  Analogous concepts can be formulated for packing dimension,
with respective notation $\PD_\star(\mu)$, $\PD^\star(\mu)$. If $\PD(\mu)$ exists, it is called the \textit{packing dimension of the measure} $\mu$; in this case it can be proved that $\PD(\mu) = \sup\{s, \bar d_\mu(x)\ge s, \text{for} \ \mu-\text{almost all} \ x \}$.   

\sp\fr The first relations between these concepts are given in the following theorem (for example \cite{Fa}, \cite{pubook}):

\bthm[General properties of dimensions of measures on metric spaces]\label{t6.6.4.}

\

 \ (i) \ If $\mu$ is a finite Borel measure on a metric space $(X,\rho)$, then
$$
\HD_\star(\mu)=\ess\inf \un d_\mu,\  \HD^\star(\mu)=\ess\sup \un
d_\mu, \ \ \text{and} \ \ \
\PD_\star (\mu)=\ess\inf\ov d_\mu, \  \PD^\star (\mu)=\ess\sup\ov d_\mu
$$

\ (ii) \ If $\mu$ is an exact dimensional finite Borel measure on a metric space $(X,\rho)$, then
both its Hausdorff dimension and packing dimension are well-defined and 
$$
\HD(\mu)=\PD(\mu)=d_\mu.
$$

\ethm

Let now $X$ be a compact connected subset of $\R^q$, $q\ge 1$ with
$X=\ov{\Int(X)}$. Consider also $E$ to be a countable set  (either finite or infinite), called an \textit{alphabet}.

\bdfn\label{rccifs}
A \textit{random countable conformal iterated function system}
$$
\cS=(\th:\La\rightarrow\La,\{\l\mapsto\varphi_e^\l\}_{e\in E})
$$ 
is defined by an invertible ergodic measure-preserving transformation of a complete probability space $(\La,{\mathcal F},m)$, namely
$$
\th:(\La,{\mathcal F},m)\rightarrow(\La,{\mathcal F},m),
$$ 
 \textit{and} by a family of injective
conformal contractions on $X$, defined for each $e \in E$ and $\lambda \in \Lambda$,
$$
\varphi_e^\l:X\rightarrow X, 
$$ 
all of whose Lipschitz constants do not exceed a common value $0<s<1$. We in fact assume that there exists a bounded open connected set $W\sbt \R^q$ containg $X$, such that all maps $\phi_e^\l:X\to X$ extend confomally to (injective) maps from $W$ to $W$.  \ $\hfill\square$
\edfn

 We will denote in the sequel by $E^\infty$ the space of one-sided infinite sequences $\omega = (\omega_1, \omega_2, \ldots), \omega_i \in E, i \ge 1$; and by $E^*$ the set of all finite sequences $\tau = (\tau_1, \tau_2, \ldots, \tau_k), \tau_i \in E, 1 \le i \le k, k \ge 1$. We also have the usual shift map $\sigma: E^\infty \to E^\infty$.
 
In the sequel assume that  the contraction maps $\varphi_e^\l:W\to W$ satisfy the following \textit{Bounded Distortion Property} (BDP):

\bproperty[BDP]\label{bdp}
There exists a function $K:[0,1)\to [1,\infty)$ such that $\lim_{t\downto 0}K(t)=K(0)=1$, and 
$$
\sup\lt\{\frac{\big|\(\phi_\om^\l\)'(y)\big|}{\big|\(\phi_\om^\l\)'(x)\big|}
:e\in E,\, \l\in\La, \, x\in X,\, ||y-x||\le t\cdot \dist(x,\R^q\sms W)\rt\}
\le K(t).
$$
\eproperty

We also require some common measurability
conditions. Precisely, we assume that for every $e\in E$ and every $x\in X$ the   map
$$
\La\ni\l\mapsto\varphi_e^\l(x)
$$
is measurable. According to Lemma~1.1 in~\cite{Crauel}, this implies that, for all $e\in E$, the maps 
$$
\La\times X\ni (\l,x)\mapsto\varphi_e(x,\l):=\varphi_e^\l(x)
$$ 
are (jointly) measurable.  
For every finite sequence $\om\in E^*$, and every $\lambda \in \La$, let us define also the (randomized) composition of contractions
\[ 
\varphi_\om^\l
:=\varphi_{\om_1}^\l\circ
  \varphi_{\om_2}^{\th(\l)}\circ\ldots\circ
  \varphi_{\om_{|\om|}}^{\th^{|\om|-1}(\l)}
\] 
This formula exhibits  the random aspect of our iterations: we
choose consecutive generators $\varphi_{\om_1},\varphi_{\om_2},\ldots,
\varphi_{\om_{n}}$ according to a random process governed by the
ergodic map $\th:\La\rightarrow\La$. This random aspect is particularly
striking if $\th$ is a Bernoulli shift when, in the random composition we choose
$\phi_e^\l$ in an independent identically distributed way.

 For $\om\in E^\infty, \lambda \in \Lambda$, we define analogously to the deterministic case (\cite{gdms},  etc.), the point
\[
\pi_\l(\omega):=\bigcap_{n=1}^\infty\varphi_{\om|_n}^\l(X),
\]
and then the fractal limit set of the random countable IFS, corresponding to $\lambda \in \Lambda$ is:
\[ 
J_\l:=\pi_\l(E^\infty)
\]

Let us denote by $\pi_\Lambda: \Lambda \times E^\infty \to \La$ and $\pi_{E^\infty}: \Lambda \times E^\infty  \to  E^\infty$,  the projections on the first, respectively the second coordinates.
And by $\pi_{\mathbb R^q}: \Lambda \times E^\infty \to \mathbb R^q$ the projection defining the limit sets $J_\lambda, \ \lambda \in \La$, namely $\pi_{\mathbb R^q}(\lambda, \omega) = \pi_\lambda(\omega)$, for $(\lambda, \omega) \in \Lambda \times E^\infty$.

Let us also denote by $\xi$  the partion of $E^\infty$ into initial cylinders of length $1$; we will work in the sequel with conditional entropies of partitions and of probability measures (see for example \cite{Wa}, \cite{Mane} for general definitions and properties). 

Given a Lebesgue space $(Y, \mathcal{B}, \mu)$ and two measurable partitions of it, $\eta$ and $\zeta$, we will sometimes write $H_\mu(\eta| \zeta)$ without loss of generality, for the measure-theoretic conditional entropy $H_\mu(\eta| \hat \zeta)$ of the partition $\eta$ with respect to the $\sigma$-algebra $\hat \zeta$ generated by $\zeta$. \
We will introduce now a notion of \textit{measure-theoretical projectional entropy} for the random infinite system and for a projection measure.

\begin{defin}\label{randomprojent}
 Given the random countable iterated function system $\mathcal{S}$ as above, and a $\theta \times \sigma$--invariant probability measure $\mu$ on $\Lambda \times E^\infty$, define the \textit{random projectional entropy} of the measure $\mu$ relative to the system $\mathcal{S}$, to be: 
$$
\hmu(\cS)
:=\H_\mu\(\pi_{E^\infty}^{-1}(\xi)\big|\pi_\La^{-1}(\e_\La)\vee(\th\times\sg)^{-1}
         (\pi_{\R^q}^{-1}(\e_{\R^q}))\)
  -\H_\mu\(\pi_{E^\infty}^{-1}(\xi)\big|\pi_\La^{-1}(\e_\La)\vee
         \pi_{\R^q}^{-1}(\e_{\R^q})\),
$$
where $\e_\La, \e_{\R^q}$ are the point partitions of $\La$, respectively $\R^q$.
\end{defin}

\

In the sequel we will consider only those $\theta \times \sigma$--invariant probability measures $\mu$ on $\Lambda \times E^\infty$ whose marginal measure on the parameter space $\La$ is equal to $m$, i. e. such that
$$
\mu\circ \pi_\La^{-1}=m
$$
We denote then by $(\mu_\l)_{\l\in\La}$ the Rokhlin's disintegration of the measure $\mu$ with respect to the fiber partition $(\pi_\La^{-1})_{\l\in\La}$. Its elements, $\{\lambda\}\times E^\infty$, $\l\in\La$, will be frequently identified with the set $E^\infty$ and we will treat each probability measure $\mu_\l$ as defined on $E^\infty$.

 The desintegration $(\mu_\l)_{\l\in\La}$ depending measurably on $\l$, is uniquely determined by the property that for any $\mu$-integrable function $g: \La \times E^\infty \to \R$, we have
$$
\int_{\Lambda \times E^\infty}gd\mu
=\int_\La\int_{E^\infty}gd\mu_\l \ dm(\l)
$$
Thus from Lemma~2.2.3 in \cite{Bogenschutz}, we have the following equivalent desintegration  formula for the random projectional entropy: 
\beq\label{1ie12}
\hmu(\cS)
=\int_\La\H_{\mu_\l}\(\xi\big|\sg^{-1}(\pi_{\th(\l)}^{-1}(\e_{J_{\th(\l)}})\)dm(\l)
  - \int_\La\H_{\mu_\l}\(\xi\big|\pi_\l^{-1}(\e_{J_\l})\)dm(\l)
\eeq
Using Definition \ref{randomprojent} and the definitions of conditional entropy and conditional expectations (for eg from \cite{Wa}, etc.), we can then further write: 
\beq\label{2ie12}
\begin{aligned}
\hmu(\cS)
=\int_\La\Big[\int_{E^\infty}\log E_{\mu_\l}&\(\1_{[\om_1]}\big|\pi_\l^{-1}(\e_{J_\l})\) 
    (\om)d\mu_\l(\om)- \\ &-\int_{E^\infty}\log E_{\mu_\l}\(\1_{[\om_1]}\big|  
     (\pi_{\th(\l)}\circ\sg)^{-1}(\e_{J_{\th(\l)}})\)(\om)d\mu_\l(\om)\Big] dm(\l)
\end{aligned}
\eeq

\

We will see that there are important differences from the finite deterministic case, since here we have a family $(J_\lambda)_{\lambda \in \Lambda}$ of possibly non-compact limit sets, and a family of boundaries at infinity $(\partial_\infty\cS_\lambda)_{\lambda \in \Lambda}$.  The $\lambda$-\textit{boundary at infinity} of $\cS$, denoted by $\cS_\lambda(\infty)$, is defined as the set of accumulation points of sequences of type $(\phi^\lambda_{e_n}(x_n))_n$, for arbitrary points $x_n \in X$ and infinitely many different indices $e_n \in E$. Similarly as in the deterministic case \cite{MU}, we define $$\cS^+_\lambda(\infty):= \mathop{\bigcup}\limits_{\omega \in E^*} \phi^{\theta(\lambda)}_\omega(\cS_\lambda(\infty))$$ 

We give now some results about the relations between the random projectional entropy $h_\mu(\cS)$ and the measure-theoretical entropy $h(\mu)$ of the $\theta\times \sigma$-invariant probability $\mu$ on $\Lambda \times E^\infty$. In this way we obtain upper and lower bounds for the random projectional entropy $h_\mu(\cS)$. 

\begin{thm}\label{ab}
In the above setting, if $\cS$ is a random countable iterated function system and if $\mu$ is a  $(\theta\times \sigma)$-invariant probability on $\Lambda \times E^\infty$, we have the following inequalities:

\ (a) $$h_\mu(\cS) \le h(\mu)$$

\ (b) Assume that there exists an integer $k \ge 1$, such that for $\mu$-almost every $(\lambda, \omega) \in \Lambda \times E^\infty$ there exists $r(\lambda, \omega) >0$ and $k$ indices $e_1, \ldots, e_k \in E$, so that if the ball $B(\pi_\lambda(\omega), r(\lambda, \omega)) \subset \R^q$ intersects a set of type $\phi^{\lambda'}_e(J_{\lambda'}), \ e \in E, \lambda' \in \Lambda$, then $e$ must belong to $\{e_1, \ldots, e_k\}$. Then
$$h_\mu(\cS) \ge h(\mu) -   \log k$$
\end{thm}

\begin{proof}

\ (a) Let us denote by $\mathcal B$ the $\sigma$-algebra of borelian sets in $\R^q$, and by $\hat \xi$ the $\sigma$-algebra  generated by the partition $\tilde \xi = \pi_{E^\infty}^{-1} \xi$ in $\Lambda \times E^\infty$.  \
We want to prove first that 
\begin{equation}\label{hatxi}
\hat \xi \vee (\theta\times \sigma)^{-1} \pi_{\R^q}^{-1}\mathcal B = \hat \xi \vee \pi_{\R^q}^{-1}
\end{equation}

\fr But an element of the $\sigma$-algebra $\hat \xi \vee (\theta\times\sigma)^{-1}\pi_{\R^q}^{-1}\mathcal B$ is a set of type
$$\mathop{\cup}\limits_{i \in E} (\Lambda \times [i]) \cap (\theta\times \sigma)^{-1}\pi_{\R^q}^{-1}A_i,$$
where $A_i \in \mathcal B, i \in E$.
Let us take an element $(\lambda, \omega) \in \pi_{\R^q}^{-1}(A_i)$, so $\pi_{\R^q}(\lambda, \omega) \in A_i$, where $\omega = (\omega_1, \omega_2, \ldots)$. Then an element $\zeta$ from the preimage set $(\theta^{-1} \times \sigma)^{-1}(\lambda, \omega)$, has the form  $(\theta^{-1}\lambda, (\omega_0, \omega_1, \ldots)$, for arbitrary $\omega_0 \in E$; if this element belongs in addition to $\Lambda \times [i]$, then $\omega_0 = i$.  Now $\pi_{\R^q}(\zeta) = \phi_i^{\theta^{-1}\lambda}(\pi_{\R^q}(\lambda, \omega) \in \phi_i^{\theta^{-1}\lambda}(A_i)$. Therefore we proved that
$$(\Lambda \times [i]) \cap (\theta \times \sigma)^{-1}\pi_{\R^q}^{-1}A_i = (\Lambda \times [i]) \cap \pi_{\R^q}^{-1}(\phi_i^{\theta^{-1}\lambda}(A_i))$$
Thus $\hat \xi \vee (\theta\times \sigma)^{-1}\pi_{\R^q}^{-1}\mathcal B \subseteq \hat \xi \vee \pi_{\R^q}^{-1}\mathcal B$, and after showing also the converse inequality of $\sigma$-algebras we obtain (\ref{hatxi}), i.e  that $\hat \xi \vee (\theta\times\sigma)^{-1}\pi_{\R^q}^{-1}\mathcal B = \hat \xi \vee \pi_{\R^q}^{-1} \mathcal B$.

 For an arbitrary integer $n \ge 1$, let us denote the measurable partition $\tilde \xi _0^{n-1}:= \xi \vee \sigma^{-1} \xi \ldots \vee 
\sigma^{-n}\xi$.
Using now the fact that  the measure $\mu$ is $(\theta\times \sigma)$-invariant on $\Lambda \times E^\infty$, and the same type of argument as in Lemma 4.8 of \cite{FH}, we obtain that for every integer $n \ge 1$, 
\begin{equation}\label{nH}
H_\mu(\tilde \xi_0^{n-1}| (\theta\times \sigma)^{-n}\pi_{\R^q}^{-1} \mathcal B) - H_\mu(\tilde \xi_0^{n-1} | \pi_{\R^q}^{-1} \mathcal B) = n \cdot \big[H_\mu(\tilde \xi| (\theta\times \sigma)^{-1}\pi_{\R^q}^{-1} \mathcal B) - H_\mu(\tilde \xi | \pi_{\R^q}^{-1} \mathcal B)\big]
\end{equation}
Hence from formula (\ref{nH}) we obtain the following inequality: $$n h_\mu(\cS) = H_\mu(\tilde \xi_0^{n-1}|(\theta\times \sigma)^{-1}\pi_{\R^q}^{-1}\mathcal B) - H_\mu(\tilde \xi_0^{n-1}|\pi_{\R^q}^{-1}\mathcal B) \le H_\mu(\tilde \xi_0^{n-1})$$
Therefore, as $h(\mu)$ is the supremum of the limits of $\frac 1n H_\mu\big(\mathop{\vee}\limits_0^{n-1} (\theta\times \sigma)^{-i} \tau\big)$ when $n \to \infty$, over all partitions $\tau$ of $\Lambda \times E^\infty$, we obtain the upper bound $h_\mu(\cS) \le h(\mu)$. 

\

\ (b) We remind that $\xi$ is the partition of $E^\infty$ into the 1-cylinders $[i]:= \{\omega \in E^\infty, \omega = (\omega_1, \omega_2, \ldots), \ \omega_1 = i\}$, for $i \in E$; and also that for simplicity of notation, given in general 2 measurable partitions $\eta, \zeta$ of a Lebesgue space $(Y, \nu)$, we will sometimes write $H_\nu(\eta| \zeta)$ instead of $H_\nu(\eta| \hat \zeta)$ where $\hat \zeta$ is the $\sigma$-algebra generated by $\zeta$. We now assume that for $\mu$-almost every $(\lambda, \omega) \in \Lambda \times E^\infty$, there are at most $k$ indices $e\in E$ so that sets of type $\phi_e^{\lambda'}(J_{\lambda'}), \lambda' \in \La$ intersect the ball $B(\pi_\lambda(\omega), r(\lambda, \omega))$. 
Let us consider next the partition $\mathcal P_n$ of $\R^q$ with sets of type $I_{(i_1, \ldots, i_q)}= [\frac{i_1}{2^n}, \frac{i+1}{2^n}) \times \ldots \times [\frac{i_q}{2^n}, \frac{i_q+1}{2^n})$, for all multi-indices $(i_1, \ldots, i_q) \in \Z^q$. 

For $m$-almost every $\lambda \in \La$ we will now construct the subpartition $\mathcal R_n(\lambda) \subseteq \mathcal P_n$, which uses only those sets $I_{(i_1, \ldots, i_q)} \in \mathcal P_n$ that contain points $\pi_\lambda(\omega) \in J_\lambda, \omega \in E^\infty$, with $r(\lambda, \omega) > q/2^n$, and where the union of all the remaining cubes $I_{(i_1, \ldots, i_q)}$ of $\mathcal P_n$ represents just one element of $\mathcal R_n(\lambda)$. 
But we assumed that for $\mu$-almost all $(\lambda, \omega) \in \Lambda \times E^\infty$, there exists a radius $r(\lambda, \omega) >0$, such  that:
\begin{equation}\label{card}
\text{Card} \{i\in E, \ \exists \lambda' \in \La \ \text{s.t} \  B(\pi_\lambda(\omega), r(\lambda, \omega)) \cap \phi_i^{\lambda'}(J_{\lambda'}) \ne \emptyset\} \le k
\end{equation}
 So using the fact that $n$ was chosen so that any cube $I_{(i_1, \ldots, i_q)} \in \mathcal R_n(\lambda)$ contains at least a point of type $\pi_\lambda(\omega), \omega \in E^\infty$ with $r(\lambda, \omega) > \frac{q}{2^n}$, we obtain that any fixed set $A$ from the partition $\pi_{\lambda}^{-1}(\mathcal R_n(\lambda))$ of $E^\infty$,  intersects \textit{at most} $k$ elements of the partition $\xi \vee \pi_{\lambda}^{-1}(\mathcal R_n(\lambda))$ of $E^\infty$. \ 
Recall also that $\mu_\lambda\circ \pi_\lambda^{-1}$ is a $\sigma$-invariant probability measure on $E^\infty$, for $\lambda \in \Lambda$. Hence from above and using \cite{Pa}, \cite{Wa}, it follows that the conditional entropy $H_{\mu_\lambda}\big(\xi| \pi_{\lambda}^{-1}(\mathcal R_n(\lambda))\big)$ satisfies:
\begin{equation}\label{H-ine}
H_{\mu_\lambda}\big(\xi| \pi_{\lambda}^{-1}(\mathcal R_n(\lambda))\big) = H_{\mu_\lambda}(\xi \vee \pi_{\lambda}^{-1} \mathcal R_n(\lambda)) - H_{\mu_\lambda}(\pi_{\lambda}^{-1}(\mathcal R_n(\lambda)) \le \log k
\end{equation}

\fr But now, since we known that for $\mu$-almost all $(\lambda, \omega) \in \La \times E^\infty$ there exists a radius $r(\lambda, \omega)>0$ satisfying condition (\ref{card}), we infer that $\pi_{\lambda}^{-1}(\mathcal R_n(\lambda)) \nearrow \pi_{\lambda}^{-1}(\epsilon_{\R^q})$, when $n \to \infty$; \ and the same conclusion for the respective $\sigma$-algebras generated by these partitions in $E^\infty$. Therefore from (\ref{H-ine}) and \cite{Pa}, and since $\mu\circ\pi_\La^{-1}= m$, it follows that for $m$-almost every $\lambda \in \La$, the conditional entropy $H_{\mu_\lambda}(\xi | \pi_{\lambda}^{-1} \mathcal B)$ satisfies the inequality $$H_{\mu_\lambda}(\xi | \pi_{\lambda}^{-1}(\mathcal B)) = \mathop{\lim}\limits_{n \to \infty} H_{\mu_\lambda}(\xi | \pi_{\lambda}^{-1} \mathcal R_n(\lambda)) \le \log k$$

\fr In addition we have that for $m$-almost any parameter $\lambda \in \La$, $$H_{\mu_\lambda}(\xi | \sigma^{-1} (\pi_{\theta(\lambda)}^{-1} \epsilon_{J_{\theta(\lambda)}})) \ge H_{\mu_\lambda}(\xi | \sigma^{-1} (\mathcal B(E^\infty))) = h_\sigma(\mu_\lambda),$$
since $\xi$ is a generator partition for $\mu_\lambda$ on $E^\infty$, and by using section 3-1 of \cite{Pa}.  
Therefore, from (\ref{1ie12}) and the last two displayed inequalities, we obtain the required inequality, namely 
$$h_\mu(\cS) \ge \int_\La h_\sigma(\mu_\lambda) d m(\lambda)  - \log k = h(\mu) - \log k$$
\end{proof}

\begin{rem}
We remark that the condition in Theorem~\ref{ab}, part (b), implies that there are no points from $\cS_\lambda(\infty)$ in any of the limit sets $J_{\lambda'}$ for all $\lambda, \lambda' \in \Lambda$. We shall give an example of such a random infinite system with overlaps in the last section. The difficulty without this condition is that, there may be a variable number of overlaps at points from the possibly non-compact fractal $J_\lambda$, and that this number may tend to $\infty$ even for a given $\lambda$, or that it may tend to $\infty$ when $\lambda$ varies in $\Lambda$; in both of these cases, we cannot obtain however a lower estimate for $h_\mu(\cS)$ like the one in Theorem \ref{ab} (b). 

\end{rem}

\

\section{Pointwise dimension for random projections of measures.}

Given a metric space $(X,\rho)$ and a measurable map $H:E^\infty\to X$, then for every sequence $\om\in E^\infty$ and every $r>0$, we shall denote by
$$
B_H(\om,r):=H^{-1}(B_\rho(H(\om),r)).
$$
 Our main result in this section is the exact dimensionality of random projections $\mu_\lambda$ on $J_\lambda$, of $(\theta\times \sigma)$-invariant probabilities $\mu$ from $\Lambda \times E^\infty$, for $m$-almost all parameters $\lambda\in \Lambda$. We start with the following:

\blem\label{l1ie19}
For all integers $k\ge 0$, every $e\in E$ and $\l\in\La$, and $\mu_\l$-a.e. $\om\in E^\infty$, we have
$$
\lim_{r\to 0}\log\frac{\mu_\l\(B_{\pi_{\th^k(\l)}\circ\sg^k}(\om,r)\cap[e]\)}
                      {\mu_\l\(B_{\pi_{\th^k(\l)}\circ\sg^k}(\om,r)\)}
=\log E_{\mu_\l}\(\1_{[e]}\big|(\pi_{\th^k(\l)}\circ\sg^k)^{-1}(\cB_{\R^q}))\)(\om).
$$
\elem

\begin{proof}
Fix $e\in E$ and define the following two Borel measures on $\R^q$:
\beq\label{1aie19}
\nu_\l:=\mu_\l\circ(\pi_{\th^k(\l)}\circ\sg^k)^{-1}, \ \text{and}
\eeq
\beq\label{1bie19}
\nu_\l^e(D):=\mu_\l\([e]\cap(\pi_{\th^k(\l)}\circ\sg^k)^{-1}(D)\), \  \  D\text{ Borel set in } \ 
  \R^d.
\eeq
Since $\nu_\l^e\le\nu_\l$, the measure $\nu_\l^e$ is absolutely continuous with respect to $\nu_\l$. Let us then define the Radon-Nikodym derivative of $\nu_\l^e$ with respect to $\nu_\l$: 
$$
g_\l^e:=\frac{d\nu_\l^e}{d\nu_\l}
$$
Then, by Theorem~2.12 in \cite{Mattila}, we have that:
\beq\label{1ie21}
g_\l^e(x)=\lim_{r\to 0}\frac{\nu_\l^e(B(x,r)}{\nu_\l(B(x,r)}
\eeq
for $\nu_\l$-a.e. $x\in\R^q$. On the other hand, for every set $F\in (\pi_{\th^k(\l)}\circ\sg^k)^{-1}\(\cB_{\R^q}\)$, say $F=(\pi_{\th^k(\l)}\circ\sg^k)^{-1}(\tilde F)$, $\tilde F\in \cB_{\R^q}$, we have
$$
\begin{aligned}
\int_F E_{\mu_\l}&\(\1_{[e]}\big|  
     (\pi_{\th^k(\l)}\circ\sg)^{-1}(\cB_{\R^q})\)d\mu_\l=\int_F\1_{[e]}d\mu_\l 
 =\mu_\l(F\cap[e]) \\
&=\mu_\l\((\pi_{\th^k(\l)}\circ\sg^k)^{-1}(\tilde F)\cap[e]\)
 =\nu_\l^e(\tilde F)
 =\int_{\tilde F}g_\l^ed\nu_\l \\
&=\int_{\tilde F}g_\l^ed(\mu_\l\circ(\pi_{\th^k(\l)}\circ\sg^k)^{-1})
 =\int_{\R^q}\1_{\tilde F}\, g_\l^ed(\mu_\l\circ(\pi_{\th^k(\l)}\circ\sg^k)^{-1})\\
&=\int_{E^\infty}\1_{\tilde F}\circ(\pi_{\th^k(\l)}\circ\sg^k)\,
        g_\l^e\circ(\pi_{\th^k(\l)}\circ\sg^k)d\nu_\l \\
&=\int_{E^\infty}\1_F\, g_\l^e\circ(\pi_{\th^k(\l)}\circ\sg^k)d\nu_\l \\
&=\int_Fg_\l^e\circ(\pi_{\th^k(\l)}\circ\sg^k)d\nu_\l.
\end{aligned}
$$
Since, in addition, both functions $E_{\mu_\l}\(\1_{[e]}\big|      (\pi_{\th^k(\l)}\circ\sg)^{-1}(\cB_{\R^q})\)$ and $g_\l^e\circ(\pi_{\th^k(\l)}\circ\sg^k)$ are non-negative and measurable with respect to the $\sg$-algebra $(\pi_{\th^k(\l)}\circ\sg)^{-1}(\cB_{\R^q})$, we conclude that
$$
g_\l^e\circ(\pi_{\th^k(\l)}\circ\sg)^{-1}(\cB_{\R^q})(\om)
= E_{\mu_\l}\(\1_{[e]}\big|(\pi_{\th^k(\l)}\circ\sg)^{-1}(\cB_{\R^q})\)(\om)
$$
for $\mu_\l$-a.e. $\om\in E^\infty$. Along with \eqref{1ie21} this means that
$$
\lim_{r\to 0}\frac{\mu_\l\(B_{\pi_{\th^k(\l)}\circ\sg^k}(\om,r)\cap[e]\)}
                      {\mu_\l\(B_{\pi_{\th^k(\l)}\circ\sg^k}(\om,r)\)}
=E_{\mu_\l}\(\1_{[\om_1]}\big|(\pi_{\th^k(\l)}\circ\sg^k)^{-1}(\e_{J_\l})\)(\om)
$$
for $\mu_\l$-a.e. $\om\in E^\infty$. Taking logarithms the lemma follows. 

\end{proof}

\bcor\label{c1ie23}
For all integers $k\ge 0$, all $\l\in\La$, and $\mu_\l$-a.e. $\om\in E^\infty$, we have
$$
\lim_{r\to 0}\log\frac{\mu_\l\(B_{\pi_{\th^k(\l)}\circ\sg^k}(\om,r)\cap[\om_1]\)}
                      {\mu_\l\(B_{\pi_{\th^k(\l)}\circ\sg^k}(\om,r)\)}
=\log E_{\mu_\l}\(\1_{[\om_1]}\big|(\pi_{\th^k(\l)}\circ\sg^k)^{-1}(\cB_{\R^q}))\).
$$
\ecor

\begin{proof}
We have 
$$
\begin{aligned}
\lim_{r\to 0}\log&\frac{\mu_\l\(B_{\pi_{\th^k(\l)}\circ\sg^k}(\om,r)\cap[\om_1]\)}
                      {\mu_\l\(B_{\pi_{\th^k(\l)}\circ\sg^k}(\om,r)\)}= \\
&=\sum_{e\in E} \1_{[e]}(\om)
\lim_{r\to 0}\log\frac{\mu_\l\(B_{\pi_{\th^k(\l)}\circ\sg^k}(\om,r)\cap[e]\)}
                      {\mu_\l\(B_{\pi_{\th^k(\l)}\circ\sg^k}(\om,r)\)} \\
&=\sum_{e\in E} \1_{[e]}(\om)\log E_{\mu_\l}\(\1_{[e]}\big|(\pi_{\th^k(\l)}\circ\sg^k)^{-1}
      (\cB_{\R^q})\)(\om) \\
&=\log E_{\mu_\l}\(\1_{[\om_1]}\big|(\pi_{\th^k(\l)}\circ\sg^k)^{-1}(\cB_{\R^q}))\)(\om).
\end{aligned}
$$

\end{proof}

\fr Now we shall prove the following.

\blem\label{l2ie23}
If $\H_\mu\( \pi_{E^\infty}^{-1}(\xi)|\pi_\La^{-1}(\e_\La)\)<\infty$, then the function
$$
g(\l,\om)
:=-\inf_{r>0}\log\frac{\mu_\l\([\om_1]\cap B_{\pi_{\th^k(\l)}\circ\sg^k}(\om,r)\)}
                      {\mu_\l\(B_{\pi_{\th^k(\l)}\circ\sg^k}(\om,r)\)}\in\R
$$
is integrable with respect to the measure $\mu$, that is it belongs to $L^1(\mu)$.
\elem

\begin{proof}
Fix $\l\in\La$. Fix also $e\in E$. As in the proof of Lemma~\ref{l1ie19} consider measures $\nu_\l$ and $\nu_\l^e$ defined by \eqref{1aie19} and  \eqref{1bie19} respectively. By Theorem~2.19 in \cite{Mattila} we have that
$$
\begin{aligned}
\nu_\l^e\Bigg(&\lt\{x\in\R^q:\inf_{r>0}\lt\{\frac{\nu_\l^e(B(x,r)}{\nu_\l(B(x,r))}\rt\}<t\rt\}\Bigg)= \\
&=\nu_\l^e\Bigg(\lt\{x\in\R^q:\sup_{r>0}\lt\{\frac{\nu_\l(B(x,r))}
      {\nu_\l^e(B(x,r))}\rt\}>1/t\rt\}\Bigg)\\
&\le C_qt\nu_\l(\R^q) =C_qt,
\end{aligned}
$$
where $1\le C_q<\infty$ is a constant depending only on $q$. What we obtained means that
$$
\mu_\l\Bigg(\lt\{\om\in E^\infty:\inf_{r>0}\lt\{\frac{\mu_\l\([e]\cap
     B_{\pi_{\th^k(\l)}\circ\sg^k}(\om,r)\)}
 {\mu_\l\(B_{\pi_{\th^k(\l)}\circ\sg^k}(\om,r)\)}\rt\}
<t\rt\}\Bigg)
\le C_qt.
$$
Let us define also the function:
$$
G_\l^e(\om):=\inf_{r>0}\lt\{\frac{\mu_\l\([e]\cap
     B_{\pi_{\th^k(\l)}\circ\sg^k}(\om,r)\)}
 {\mu_\l\(B_{\pi_{\th^k(\l)}\circ\sg^k}(\om,r)\)}\rt\}.
$$
Then the previous inequality can be rewritten as:
$$
\mu_\l\((G_\l^e)^{-1}([0,t))\)\le C_qt.
$$
Define now the function $g_\l:E^\infty\to\R$ by
$
g_\l(\om)=g(\l,\om)$.
Thus the following equlity holds:
$$
g_\l=\sum_{e\in E}-\1_{[e]}\log G_\l^e.
$$
Noting also that $g_\l\ge 0$, we obtain therefore:
$$
\begin{aligned}
\int_{E^\infty}g_\l d\mu_\l
&=\sum_{e\in E}-\int_{[e]}\log G_\l^e d\mu_\l
 =\sum_{e\in E}\int_0^\infty\mu_\l\(\{\om\in [e]:-\log G_\l^e(\om)>s\}\)ds \\
&=\sum_{e\in E}\int_0^\infty\mu_\l\(\{\om\in [e]:G_\l^e(\om)<e^{-s}\}\)ds \\
&=\sum_{e\in E}\int_0^\infty\mu_\l\(\{\om\in E^\infty:G_\l^e(\om)<e^{-s}\}\cap [e]\)ds \\
&\le\sum_{e\in E}\int_0^\infty\min\{\mu_\l([e]),Cqe^{-s}\}ds \\
&=\sum_{e\in E}\lt(\int_0^{-\log\mu_\l([e])+\log C_q}\mu_\l([e])ds
    +\int_{-\log\mu_\l([e])+\log C_q}Cqe^{-s}ds\rt)  \\
&=\sum_{e\in E}\(-\mu_\l([e])\log\mu_\l([e])+\log(C_q)\mu_\l([e]))+\mu_\l([e])\) \\
&=1+\log(C_q)+\sum_{e\in E}\(-\mu_\l([e])\log\mu_\l([e])\) \\
&=1+\log(C_q)+\H_{\mu_l}(\xi)
\end{aligned}
$$
Since $\H_\mu\( \pi_{E^\infty}^{-1}(\xi)|\pi_\La^{-1}(\e_\La)\)<\infty$, it therefore follows from Lemma~2.3 in \cite{Bogenschutz} that 
$$
\begin{aligned}
\int_{\La\times E^\infty}gd\mu
&=\int_{\La}\int_{E^\infty}g_\l d\mu_\l dm(\l)
  \le 1+\log(C_q)+\int_{\La}\H_{\mu_l}(\xi) \\
&= 1+\log(C_q)+\H_\mu\( \pi_{E^\infty}^{-1}(\xi)|\pi_\La^{-1}(\e_\La)\) <\infty
\end{aligned}
$$
The proof is thus finished. 

\end{proof}

\begin{rem}\label{finiteent}
\fr We assumed above the finite entropy condition $\H_\mu\( \pi_{E^\infty}^{-1}(\xi)|\pi_\La^{-1}(\e_\La)\)<\infty$. This is not a restrictive condition, and it is satisfied by many measures and systems. For example, it is clearly satisfied if the alphabet $E$ is finite. More interestingly, it is also satisfied when $E$ is infinite and  $\mu = m \times \nu$, where $m$ is an arbitrary $\theta$-invariant probability on $\Lambda$, and $\nu$ is a $\sigma$-invariant probability on $E^\infty$ satisfying $\nu([i]) = \nu_i,  i \in E$ and $$h(\nu) = -\mathop{\sum}\limits_{i \in E} \nu_i \log \nu_i < \infty$$
Indeed, if $\mathcal A$ is the $\sigma$-algebra generated in $\Lambda \times E^\infty$ by the partition $ \pi_\Lambda^{-1}(\epsilon_\Lambda)$, and if $\tilde \xi:= \pi_{E^\infty}^{-1} \xi$, then $H_\mu(\tilde \xi|\mathcal A) = \int I_\mu(\tilde \xi|\mathcal A), \ \text{where} \ I_\mu(\tilde \xi| \mathcal A) \ \text{is the information function}$
$$I_\mu(\tilde \xi| \mathcal A) := - \mathop{\sum}\limits_{A \in \tilde \xi} \chi_A \cdot \log E_\mu(\chi_A| \mathcal A)$$

\fr Now, the conditional expectation $E_\mu(\chi_A|\mathcal A) =: g_A$ is $\mathcal A$-measurable, and $\int_{B \times E^\infty} g d \mu = \int_{B \times E^\infty} \chi_A d\mu$, for all sets $B$ measurable in $\Lambda$. Hence if $A = \Lambda \times [i]$, then $\int g_A d\mu = \mu(A \cap (B \times E^\infty)) = m(B) \cdot \nu_i$, so $g_A = \nu_i$  and $H_\mu(\tilde \xi|\mathcal A) = - \mathop{\sum}\limits_{i \in E} \nu_i \log \nu_i$.  Therefore,  if $h(\nu) < \infty$, then
$$H_\mu(\tilde \xi|\mathcal A) < \infty$$
$\hfill\square$
\end{rem}

\fr As an immediate consequence of Lemma \ref{l2ie23}, Corollary~\ref{c1ie23}, and Lebesgue's Dominated Convergence Theorem, we get the following:

\blem\label{l1ie27}
If $\H_\mu\( \pi_{E^\infty}^{-1}(\xi)|\pi_\La^{-1}(\e_\La)\)<\infty$, then 
$$
\lim_{r\to 0}
\log\frac{\mu_\l\([\om_1]\cap B_{\pi_{\th^k(\l)}\circ\sg^k}(\om,r)\)}
                      {\mu_\l\(B_{\pi_{\th^k(\l)}\circ\sg^k}(\om,r)\)}
=\log E_{\mu_\l}\(\1_{[\om_1]}\big|(\pi_{\th^k(\l)}\circ\sg^k)^{-1}(\cB_{\R^q}))\)(\om)
$$
for $\mu$-a.e. $(\l,\om)\in \La\times E^\infty$, and the convergence holds also in $L^1(\mu)$.
\elem

\fr Now we shall prove the following:

\blem\label{l1ie13}
For every $K\ge 1$ there exists $R_1>0$ such that
$$
[\om_1]\cap B_{\pi_\l}\(\om,K\big|\(\phi_{\om_1}^\l\)'(\pi_{\th(\l)}(\sg(\om)))\big|r\)
\spt [\om_1]\cap B_{\pi_{\th(\l)}\circ\sg}(\om,r)
$$
for all $\l\in\La$, all $\om\in E^\infty$, and all $r\in [0,R_1]$.
\elem

\begin{proof}
Let $\tau\in B_{\pi_{\th(\l)}\circ\sg}(\om,r)$. Then $\tau_1=\om_1$ and $\pi_{\th(\l)}(\sg(\tau))\in B(\pi_{\th(\l)}(\sg(\om)),r)$. Hence, 
$$
\begin{aligned}
\pi_\l(\tau)
&=\phi_{\om_1}^\l\(\pi_{\th(\l)}(\sg(\tau))\)
\in \phi_{\om_1}^\l\(B(\pi_{\th(\l)}(\sg(\om)),r)\) \\
&\sbt B\(\phi_{\om_1}^\l\(\pi_{\th(\l)}(\sg(\om))\), 
  K\big|\(\phi_{\om_1}^\l\)'(\pi_{\th(\l)}(\sg(\om)))\big|r\) \\
&= B\(\pi_\l(\om),K\big|\(\phi_{\om_1}^\l\)'(\pi_{\th(\l)}(\sg(\om)))\big|r\),
\end{aligned}
$$
where, because of the Bounded Distortion Property (BDP), the inclusion sign "$\sbt$" holds assuming $r>0$ to be small enough. This means that
$$
\tau\in \pi_\l^{-1}\(B\(\pi_\l(\om),K\big|\(\phi_{\om_1}^\l\)'(\pi_{\th(\l)}(\sg(\om)))\big|r\)\)
=B_{\pi_\l}\(\om,K\big|\(\phi_{\om_1}^\l\)'(\pi_{\th(\l)}(\sg(\om)))\big|r\)
$$
Since also already know that $\tau_1=\om_1$, we are thus done.

\end{proof}

\blem\label{l2ie13}
For every $K\ge 1$ there exists $R_2>0$ such that
$$
[\om_1]\cap B_{\pi_\l}\(\om,K^{-1}\big|\(\phi_{\om_1}^\l\)'(\pi_{\th(\l)}(\sg(\om)))
 \big|r\)\sbt [\om_1]\cap B_{\pi_{\th(\l)}\circ\sg}(\om,r)
$$
for all $\l\in\La$, all $\om\in E^\infty$, and all $r\in [0,R_2]$.
\elem

\begin{proof}
Because of the Bounded Distortion Property (BDP), we have for all $r\ge 0$ small enough, say $0\le r\le R_2$, that
$$
\begin{aligned}
B_{\pi_\l}\(\om,K^{-1}\big|\(\phi_{\om_1}^\l\)'(\pi_{\th(\l)}(\sg(\om)))\big|r\)
&=\pi_\l^{-1}\(B\(\pi_\l(\om),K^{-1}\big|\(\phi_{\om_1}^\l\)'(\pi_{\th(\l)}(\sg(\om)))
\big|r\)\) \\
&\sbt \pi_\l^{-1}\(\phi_{\om_1}^\l\(B(\pi_{\th(\l)}(\sg(\om)),r)\)\)
\end{aligned}
$$
So, fixing $\tau\in [\om_1]\cap B_{\pi_\l}\(\om,K^{-1}\big|\(\phi_{\om_1}^\l\)'(\pi_{\th(\l)}(\sg(\om)))\big|r\)$, we have $\tau_1=\om_1$ and 
$$
\pi_\l(\tau)
=\phi_{\om_1}^\l\(\pi_{\th(\l)}(\sg(\tau))\)
\phi_{\om_1}^\l\(B(\pi_{\th(\l)}(\sg(\om)),r)\).
$$
This means that $\pi_{\th(\l)}(\sg(\tau))\in B(\pi_{\th(\l)}(\sg(\om)),r)\)$, or equivalently, $\tau\in B_{\pi_{\th(\l)}\circ\sg}(\om,r)$. The required inclusion is thus proved and the proof is complete.

\end{proof}

\fr Since the measure $\mu$ is fiberwise invariant, we have for all $\om\in E^\infty$, all $r>0$, and $m$-a.e. $\l\in\La$ that
\beq\label{1ie15}
\begin{aligned}
\mu_\l\(B_{\pi_{\th(\l)}\circ \sg}(\om),r)\)
&=\mu_\l\((\pi_{\th(\l)}\circ\sg)^{-1}\(B(\pi_{\th(\l)}\circ\sg(\om),r)\)\) \\
&=\mu_\l\circ\sg^{-1}\(\pi_{\th(\l)}^{-1}\(B\(\pi_{\th(\l)}(\sg(\om)),r)\)\) \\
&=\mu_{\th(\l)}\(B_{\pi_{\th(\l)}}(\sg(\om),r)\)
\end{aligned}
\eeq

As an immediate consequence of this formula along with Lemma~\ref{l1ie13} and Lemma~\ref{l2ie13}, we get the following:

\blem\label{l1ie15}
For every $K>1$ there exists $R_K>0$ such that
$$
\frac{\mu_\l\([\om_1]\cap B_{\pi_\l}\(\om,K\big|\(\phi_{\om_1}^\l\)'(\pi_{\th(\l)}(\sg(\om)))\big|r\)\)}{\mu_{\th(\l)}\(B_{\pi_{\th(\l)}}(\sg(\om),r)\)}
\ge \frac{\mu_\l\([\om_1]\cap B_{\pi_{\th(\l)}\circ \sg}(\om,r)\)}  {\mu_\l\(B_{\pi_{\th(\l)}\circ \sg}(\om),r)\)}
$$
and
$$
\frac{\mu_\l\([\om_1]\cap B_{\pi_\l}\(\om,K^{-1}\big|\(\phi_{\om_1}^\l\)'(\pi_{\th(\l)}(\sg(\om)))\big|r\)\)}{\mu_{\th(\l)}\(B_{\pi_{\th(\l)}}(\sg(\om),r)\)}
\le \frac{\mu_\l\([\om_1]\cap B_{\pi_{\th(\l)}\circ \sg}(\om,r)\)} {\mu_\l\(B_{\pi_{\th(\l)}\circ \sg}(\om,r)\)}
$$
for all $\om\in E^\infty$, all $r\in(0,R_K]$, and $m$-a.e. $\l\in\La$.
\elem

\

\blem\label{l1ie14}
We have that
$$
\int_{\La}\int_{E^\infty}\log\mu_\l\(B_{\pi_\l}\(\om,r)\)d\mu_\l(\om)dm(\l)>-\infty
$$
for all $r>0$.
\elem

\begin{proof}
Since $X$ is compact there exist finitely many points $z_1, z_2,\ld, z_l$ in $X$ such that
$$
\bu_{j=1}^l B(z_j,r/2)\spt X.
$$
For every $\l\in\La$ and every integer $n\ge 0$ define the set of sequences:
$$
A_n(\l):=\{\om\in E^\infty:e^{-(n+1)}<\mu_\l\(B_{\pi_\l}\(\om,r)\)\le e^{-n}\}.
$$
Assume that 
$$
A_n(\l)\cap \pi_\l^{-1}(B(z_j,r/2))\ne\es
$$
for some $1\le j\le l$. Fix $\g\in A_n(\l)\cap \pi_\l^{-1}(B(z_j,r/2))$ arbitrary. Then, because of the triangle inequality, $\pi_\l^{-1}(B(z_j,r/2))\sbt B_{\pi_\l}(\om,r)$. Therefore,
$$
\mu_\l\(A_n(\l)\cap\pi_\l^{-1}(B(z_j,r/2))\)
\le \mu_\l\(\pi_\l^{-1}(B(z_j,r/2))\)
\le \mu_\l
\(B_{\pi_\l}(\om,r)\)
\le e^{-n}.
$$
However,  $\mu_\l(A_n(\l))\cap\pi_\l^{-1}(B(z_i,r/2))\)=0\le e^{-n}$ if $A_n(\l)\cap \pi_\l^{-1}(B(z_i,r/2))=\es$, for some $1\le i\le l$. Hence, 
since $\lt\{\pi_\l^{-1}(B(z_j,r/2))\rt\}_{j=1}^l$ is a cover of $E^\infty$, this implies that
$$
\mu_\l(A_n(\l))\le le^{-n}
$$
Therefore we obtain,
$$
\begin{aligned}
\int_{E^\infty}-\log\mu_\l\(B_{\pi_\l}\(\om,r)\)d\mu_\l(\om)
&=   \sum_{n=0}^\infty\int_{A_n(\l)}-\log\mu_\l\(B_{\pi_\l}\(\om,r)\)d\mu_\l(\om)\\
&\le \sum_{n=0}^\infty(n+1)le^{-n} =l\sum_{n=0}^\infty(n+1)e^{-n} 
<\infty.
\end{aligned}
$$
Hence, from the above, we can conclude that
$$
\int_{\La}\int_{E^\infty}\log\mu_\l\(B_{\pi_\l}\(\om,r)\) \ d\mu_\l(\om)dm(\l)
\le l\sum_{n=0}^\infty(n+1)e^{-n} 
<\infty.
$$
\end{proof}
\fr Then employing this lemma and Birkhoff's Ergodic Theorem, we obtain the following:
\blem\label{l2ie14}
For all $r>0$ and $\mu$-a.e. $(\l,\om)\in\La\times E^\infty$, we have:
$$
\lim_{n\to\infty}\frac1n\log\mu_{\th^n(\l)}\(B_{\pi_{\th^n(\l)}}\(\sg^n(\om),r)\)=0
$$
\elem

\fr Now, we shall prove the following:

\blem\label{l2ie15}
If $\H_\mu\(\pi_{E^\infty}^{-1}(\xi)|\pi_\La^{-1}(\e_\La)\)<\infty$, then for every $K>1$, all $r\in(0,R_K)$ and $\mu$-a.e. $(\l,\om)\in\La\times E^\infty$, we have that
\beq\label{1ie17}
\varlimsup_{n\to\infty}\frac1n\log\mu_\l\(B_{\pi_\l}\(\om,K^{-n}
\big|\(\phi_{\om|_n}^\l\)'(\pi_{\th^n(\l)}(\sg^n(\om)))\big|r\)\)
\le -\hmu(\cS), \
\eeq
and moreover
\beq\label{2ie17}
\varliminf_{n\to\infty}\frac1n\log\mu_\l\(B_{\pi_\l}\(\om,K^n
\big|\(\phi_{\om|_n}^\l\)'(\pi_{\th^n(\l)}(\sg^n(\om)))\big|r\)\)
\ge -\hmu(\cS).
\eeq
\elem

\begin{proof}
We prove the first inequality by relying on the second inequality of Lemma~\ref{l1ie15}. The proof of the second inequality of the lemma is analogous and will be omitted. We have:
$$
\begin{aligned}
T&_{\l,n}^-(\om)= \\
:&=\log\mu_\l\(B_{\pi_\l}\(\om, K^{-n}
         \big|\(\phi_{\om|_n}^\l\)'(\pi_{\th^n(\l)} (\sg^n(\om)))\big|r\)\) \\
&=\sum_{j=0}^{n-1}\log\frac{\mu_{\th^j(\l)}\(B_{\pi_{\th^j(\l)}}\(\sg^j(\om), K^{-(n-j)}
    \big|\(\phi_{\sg^j(\om)|_{n-j}}^{\th^j(\l)}\)'(\pi_{\th^n(\l)}(\sg^n(\om)))\big|r\)\)}
  {\mu_{\th^j(\l)}\(B_{\pi_{\th^{j+1}(\l)}}\(\sg^{j+1}(\om), K^{-(n-(j+1))}
    \big|\(\phi_{\sg^{j+1}(\om)|_{n-(j+1)}}^{\th^{j+1}(\l)}\)'
   (\pi_{\th^n(\l)}(\sg^n(\om)))\big|r\)\)} + \\
& \  \  \  \  \  \   \  \  \  \  \  \  \ + 
    \log\mu_{\th^n(\l)}\(B_{\pi_{\th^n(\l)}}\(\sg^n(\om),r)\) \\
&=\sum_{j=0}^{n-1}\log\frac{\mu_{\th^j(\l)}\([(\sg^j(\om))_1]\cap 
    B_{\pi_{\th^j(\l)}}\(\sg^j(\om), K^{-(n-j)}
    \big|\(\phi_{\sg^j(\om)|_{n-j}}^{\th^j(\l)}\)'(\pi_{\th^n(\l)}(\sg^n(\om)))\big|r\)\)}
  {\mu_{\th^j(\l)}\(B_{\pi_{\th^{j+1}(\l)}}\(\sg^{j+1}(\om), K^{-(n-(j+1))}
    \big|\(\phi_{\sg^{j+1}(\om)|_{n-(j+1)}}^{\th^{j+1}(\l)}\)'
   (\pi_{\th^n(\l)}(\sg^n(\om)))\big|r\)\)} - \\
& \ -\sum_{j=0}^{n-1}\log\frac{\mu_{\th^j(\l)}\([(\sg^j(\om))_1]\cap 
    B_{\pi_{\th^j(\l)}}\(\sg^j(\om), K^{-(n-j)}
    \big|\(\phi_{\sg^j(\om)|_{n-j}}^{\th^j(\l)}\)'(\pi_{\th^n(\l)}(\sg^n(\om)))\big|r\)\)}
{\mu_\l\(B_{\pi_{\th^j(\l)}}\(\sg^j(\om), K^{-(n-j)}
\big|\(\phi_{\sg^j(\om)|_{n-j}}^{\th^j(\l)}\)'(\pi_{\th^n(\l)}(\sg^n(\om)))\big|r\)\)}+\\
& \  \  \  \  \  \   \  \  \  \  \  \  \ + 
    \log\mu_{\th^n(\l)}\(B_{\pi_{\th^n(\l)}}\(\sg^n(\om),r)\) \\
&\le\sum_{j=0}^{n-1}\log\frac{\mu_{\th^j(\l)}\([(\sg^j(\om))_1]\cap 
   B_{\pi_{\th^{j+1}(\l)}\circ\sg}\(\sg^{j+1}(\om), K^{-(n-(j+1))}
     \big|\(\phi_{\sg^{j+1}(\om)|_{n-(j+1)}}^{\th^{j+1}(\l)}\)'
     (\pi_{\th^n(\l)}(\sg^n(\om)))\big|r\)\)}
{\mu_{\th^j(\l)}\(B_{\pi_{\th^{j+1}(\l)}\circ\sg}\(\sg^{j+1}(\om), K^{-(n-(j+1))}
    \big|\(\phi_{\sg^{j+1}(\om)|_{n-(j+1)}}^{\th^{j+1}(\l)}\)'
   (\pi_{\th^n(\l)}(\sg^n(\om)))\big|r\)\)} -  \\
& -\sum_{j=0}^{n-1}\log\frac{\mu_{\th^j(\l)}\([(\sg^j(\om))_1]\cap 
    B_{\pi_{\th^j(\l)}}\(\sg^j(\om), K^{-(n-j)}
    \big|\(\phi_{\sg^j(\om)|_{n-j}}^{\th^j(\l)}\)'(\pi_{\th^n(\l)}(\sg^n(\om)))\big|r\)\)}
{\mu_\l\(B_{\pi_{\th^j(\l)}}\(\sg^j(\om), K^{-(n-j)}
\big|\(\phi_{\sg^j(\om)|_{n-j}}^{\th^j(\l)}\)'(\pi_{\th^n(\l)}(\sg^n(\om)))\big|r\)\)}+\\
& \  \  \  \  \  \   \  \  \  \  \  \  \ + 
    \log\mu_{\th^n(\l)}\(B_{\pi_{\th^n(\l)}}\(\sg^n(\om),r)\) \\
&=\sum_{j=0}^{n-1}W_{n-j}^-((\th\times\sg)^j(\l,\om))
   -\sum_{j=0}^{n-1}G_{n-j}^-((\th\times\sg)^j(\l,\om))
    +\log\mu_{\th^n(\l)}\(B_{\pi_{\th^n(\l)}}\(\sg^n(\om),r)\),
\end{aligned}
$$
where for all $i\ge 1$,
$$
W_i^-(\l,\om)
:=\log\frac{\mu_\l\([\om_1]\cap B_{\pi_{\th(\l)}\circ\sg}\(\sg^{j+1}(\om), 
  K^{-(i-1)}\big|\(\phi_{\sg(\om)|_{i-1}}^{\th(\l)}\)'
     (\pi_{\th^i(\l)}(\sg^i(\om)))\big|r\)\)}
{\mu_\l\(B_{\pi_{\th(\l)}\circ\sg}\(\sg^{j+1}(\om), K^{-(i-1)}
 \big|\(\phi_{\sg(\om)|_{i-1}}^{\th(\l)}\)'(\pi_{\th^i(\l)}(\sg^i(\om)))
     \big|r\)\)}
$$
and where
$$
G_i^-(\l,\om)
:=\log
  \frac{\mu_\l\([\om_1]\cap B_{\pi_\l}\(\om, K^{-i}
       \big|\(\phi_{\om|_i}^\l\)'(\pi_{\th^i(\l)} (\sg^i(\om)))\big|r\)\)}
{\mu_\l\(B_{\pi_\l}\(\om, K^{-i}
       \big|\(\phi_{\om|_i}^\l\)'(\pi_{\th^i(\l)} (\sg^i(\om)))\big|r\)\)}.
$$

\ 

Now, by virtue of Lemma~\ref{l1ie27} we see that Corollary~1.6, p. 96 in \cite{Mane}, applies to the sequences $(W_i^-)_{i=1}^\infty$ and $(W_i^-)_{i=1}^\infty$. This, in conjunction with Lemma \ref{l1ie27}, Lemma~\ref{l2ie14}, the ergodicity of the measure $\mu$ with respect to the dynamical system $\th\times\sg$, and formula \eqref{2ie12}, gives us the following inequalities:
$$
\begin{aligned}
\varlimsup_{n\to\infty}T_{\l,n}^-(\om)
&\le\int_{\La\times E^\infty}\Big(\log E_{\mu_\l}\(\1_{[\om_1]}\big|(\pi_{\th(\l)}\circ\sg)^{-1}(\cB_{\R^q}))\)(\om) -  \\
& \  \  \  \   \  \  \  \  \  \   \  \  \ -\log E_{\mu_\l}\(\1_{[\om_1]}\big|(\pi_\l^{-1}(\cB_{\R^q}))\)(\om)\Big)d\mu_\l(\om)dm(\l) \\
&=-\hmu(\cS).
\end{aligned}
$$
This finishes thus the proof.

\end{proof}

\begin{defin}
In the above setting, let us define the \textit{Lyapunov exponent} of the measure $\mu$ with respect to the endomorphism $\th\times\sg:\La\times E^\infty\to\La\times E^\infty$ and the random countable iterated function system $\mathcal{S}$:
$$
\chi_\mu:=\int_{\La\times E^\infty}
-\log\big|\(\phi_{\om_1}^\l\)'(\pi_{\th(\l)} (\sg(\om)))\big|d\mu(\l,\om).
$$
\end{defin}

Since the above dynamical system is ergodic, then Birkhoff's Ergodic Theorem yields that, for $\mu$-a.e. $(\l,\om)\in\La\times E^\infty$, we have
\begin{equation}\label{l2ie27}
\lim_{n\to\infty}\frac1n\log
\big|\(\phi_{\om|_n}^\l\)'(\pi_{\th^n(\l)}(\sg^n(\om)))\big|
=\chi_\mu.
\end{equation}

\

\fr As a consequence of this lemma and Lemma~\ref{l2ie15}, we now prove the main result of our paper:

\bthm\label{l3ie27}
If $\H_\mu\(\pi_{E^\infty}^{-1}(\xi)|\pi_\La^{-1}(\e_\La)\)<\infty$, then for
$\mu$-a.e. $(\l,\om)\in\La\times E^\infty$, we have
$$
\lim_{r\to 0}\frac{\log\(\mu_\l\circ\pi_\l^{-1}\(B_{\pi_\l}(\om,r)\)\)}{\log r}
=\frac{\hmu(\cS)}{\chi_\mu}.
$$
\ethm

\begin{proof}
What we want to prove is that:
$$
\lim_{r\to 0}\frac{\log\mu_\l\(B_{\pi_\l}(\om,r)\)}{\log r}
=\frac{\hmu(\cS)}{\chi_\mu}.
$$
Fix $K>1$. Fix also $(\l,\om)\in\La\times E^\infty$. Consider any $r\in\(0,K^{-1} \big|\(\phi_{\om_1}^\l\)'(\pi_{\th(\l)} (\sg(\om)))\big|\)$. There then exists a largest $n\ge 0$ such that
$$
r\le K^{-n}\big|\(\phi_{\om|_n}^\l\)'(\pi_{\th^n(\l)}(\sg^n(\om)))\big|R_K.
$$
Then for $n\ge 1$,
$$
B_{\pi_\l}(\om,r)
\sbt B_{\pi_\l}\(\om,K^{-n}\big|\(\phi_{\om|_n}^\l\)'(\pi_{\th^n(\l)}(\sg^n(\om)))\big|R_K\), \ \ \text{and} 
$$
$$
r\ge K^{-(n+1)}\big|\(\phi_{\om|_{n+1}}^\l\)'(\pi_{\th^{n+1}(\l)}(\sg^{n+1}(\om)))\big|R_K.
$$
Therefore,
$$
\begin{aligned}
\frac{\log\mu_\l(\(B_{\pi_\l}(\om,r)\)}{\log r}
&\ge \frac{\log\mu_\l(B_{\pi_\l}\(\om,K^{-n}\big|\(\phi_{\om|_n}^\l\)'(\pi_{\th^n(\l)}(\sg^n(\om)))\big|R_K\)\)}{\log r} \\
&\ge \frac{\log\mu_\l(B_{\pi_\l}\(\om,K^{-n}\big|\(\phi_{\om|_n}^\l\)'(\pi_{\th^n(\l)}(\sg^n(\om)))\big|R_K\)\)}
{-(n+1)\log K+\log\big|\(\phi_{\om|_{n+1}}^\l\)'(\pi_{\th^{n+1}(\l)}(\sg^{n+1}(\om)))\big|+\log R_K} \\
&= \frac{\frac1n\log\mu_\l(B_{\pi_\l}\(\om,K^{-n}\big|\(\phi_{\om|_n}^\l\)'(\pi_{\th^n(\l)}(\sg^n(\om)))\big|R_K\)\)}
{-(1+\frac1n)\log K+\frac1n\log\big|\(\phi_{\om|_{n+1}}^\l\)'(\pi_{\th^{n+1}(\l)}(\sg^{n+1}(\om)))\big|+\frac1n\log R_K}.
\end{aligned}
$$
Hence, applying formula \eqref{1ie17} from Lemma~\ref{l2ie15}, and also Lemma~\ref{l2ie27}, we get
$$
\varliminf_{r\to 0}\frac{\log\mu_\l(\(B_{\pi_\l}(\om,r)\)}{\log r}
\ge \frac{\hmu(\cS)}{\log K+\chi_\mu}
$$
for all $(\l,\om)$ in some measurable set $\Om_K^+\sbt \La\times E^\infty$ with $\mu(\Om_K^+)=1$. Then 
$$
\mu\lt(\Om^+:=\bi_{j=1}^\infty \Om_{\frac{j+1}j}^+\rt)=1
$$
and 
\beq\label{1ie29}
\varliminf_{r\to 0}\frac{\log\mu_\l(\(B_{\pi_\l}(\om,r)\)}{\log r}
\ge \frac{\hmu(\cS)}{\chi_\mu}
\eeq
for all $(\l,\om)\in\Om^+$. For the proof of the opposite direction fix any $K>1$ so small that
\beq\label{2ie29}
K^{-1}>\ess\sup\big\{\big|\big|\(\phi_e^\l\)'\big|\big|:e\in E, \l\in\La\big\}.
\eeq
Having $(\l,\om)\in\La\times E^\infty$ fix any $r\in\(0,KR_K\ess\sup\big\{\big|\big|\(\phi_e^\l\)'\big|\big|:e\in E, \l\in\La\big\}$. Because of \eqref{2ie29} there exists a least $n\ge 1$ such that 
$$
K^n\big|\(\phi_{\om|_n}^\l\)'(\pi_{\th^n(\l)}(\sg^n(\om)))\big|R_K\le r.
$$
Then, because of our choice of $r$, we have that $n\ge 2$,
$$
K^{n-1}\big|\(\phi_{\om|_{n-1}}^\l\)'(\pi_{\th^{n-1}(\l)}(\sg^{n-1}(\om)))\big|R_K\le r,
$$
and 
$$
B_{\pi_\l}(\om,r)
\spt B_{\pi_\l}\(\om,K^n\big|\(\phi_{\om|_n}^\l\)'(\pi_{\th^n(\l)}(\sg^n(\om)))\big|R_K\).
$$
Therefore,
$$
\begin{aligned}
\frac{\log\mu_\l(\(B_{\pi_\l}(\om,r)\)}{\log r}
&\le \frac{\log\mu_\l(B_{\pi_\l}\(\om,K^n\big|\(\phi_{\om|_n}^\l\)'(\pi_{\th^n(\l)}(\sg^n(\om)))\big|R_K\)\)}{\log r} \\
&\ge \frac{\log\mu_\l(B_{\pi_\l}\(\om,K^n\big|\(\phi_{\om|_n}^\l\)'(\pi_{\th^n(\l)}(\sg^n(\om)))\big|R_K\)\)}
{(n-1)\log K+\log\big|\(\phi_{\om|_{n-1}}^\l\)'(\pi_{\th^{n-1}(\l)}(\sg^{n-1}(\om)))\big|+\log R_K} \\
&= \frac{\frac1n\log\mu_\l(B_{\pi_\l}\(\om,K^n\big|\(\phi_{\om|_n}^\l\)'(\pi_{\th^n(\l)}(\sg^n(\om)))\big|R_K\)\)}
{-(1-\frac1n)\log K+\frac1n\log\big|\(\phi_{\om|_{n-1}}^\l\)'(\pi_{\th^{n-1}(\l)}(\sg^{n-1}(\om)))\big|+\frac1n\log R_K}.
\end{aligned}
$$
Hence, applying formula \eqref{1ie17} from Lemma~\ref{l2ie15}, and also Lemma~\ref{l2ie27}, we get
$$
\varlimsup_{r\to 0}\frac{\log\mu_\l(\(B_{\pi_\l}(\om,r)\)}{\log r}
\le \frac{\hmu(\cS)}{-\log K+\chi_\mu}
$$
for all $(\l,\om)$ in some measurable set $\Om_K^-\sbt \La\times E^\infty$ with $\mu(\Om_K^-)=1$. Then we have:
$$
\mu\lt(\Om^-:=\bi_{j=k}^\infty \Om_{\frac{j+1}j}^-\rt)=1,
$$
where $k\ge 1$ is taken to be so large that $\frac{k+1}k\ess\sup\big\{\big|\big|\(\phi_e^\l\)'\big|\big|:e\in E, \l\in\La\big\}<1$. Also, 
$$
\varlimsup_{r\to 0}\frac{\log\mu_\l(\(B_{\pi_\l}(\om,r)\)}{\log r}
\le \frac{\hmu(\cS)}{\chi_\mu}
$$
for all $(\l,\om)\in\Om^-$. Along with \eqref{1ie29} this yields $\mu\(\Om^+\cap\Om^-\)=1$ and moreover,
$$
\lim_{r\to 0}\frac{\log\mu_\l(\(B_{\pi_\l}(\om,r)\)}{\log r}
=\frac{\hmu(\cS)}{\chi_\mu}
$$
for all $(\l,\om)\in\Om^+\cap\Om^-$, which gives therefore the required  dimensional exactness.
\end{proof}

From the above Theorem \ref{l3ie27} and Theorem \ref{t6.6.4.} we obtain the following result, giving the (common) Hausdorff dimension and packing dimension of the projections $\mu_\lambda\circ \pi_\lambda^{-1}$ on the random limit sets $J_\lambda$: 

\begin{cor}\label{ptwdim}
In the above setting  if $\mu$ is a $\theta\times \sigma$-invariant probability on $\La \times E^\infty$ whose marginal on $\La$ is $m$, and if $\H_\mu\(\pi_{E^\infty}^{-1}(\xi)|\pi_\La^{-1}(\e_\La)\)<\infty$, then for $m$-a.e $\lambda \in \La$, we have $$
\HD(\mu_\lambda \circ \pi_\lambda^{-1}) 
= \PD(\mu_\lambda \circ \pi_\lambda^{-1}) 
= \frac{h_\mu(\mathcal{S})}{\chi_\mu}
$$
\end{cor}

\

\section{Classes of random countable IFS with overlaps.} 

In this section we will study several classes of examples of random countable IFS with overlaps, and their invariant measures.

\

\textbf{4.1. Randomizations related to Kahane-Salem sets.} \

\

In \cite{KS} Kahane and Salem studied the convolution of infinitely many Bernoulli distributions, namely the measure $\mu= B(\frac{x}{ r_0}) * B(\frac{x}{r_1})* \ldots$, where $B(x)$ denotes the Bernoulli probability supported only at the points $-1, +1$ and giving measure $\frac 12$ to each one of them. The support of $\mu$ is the set $F$ of points of the form $\epsilon_0 r_0 + \epsilon_1 r_1 + \ldots$, where $\epsilon_k$ is equal to $+1$ or $-1$ with equal probabilities. If we assume  $\mathop{\sum}\limits_0^\infty r_k = 1$, and if we introduce the sequence $(\rho_n)_{n\ge 0}$ defined by 
$$
r_0 = 1-\rho_0, r_1 = \rho_0(1-\rho_1), r_2 = \rho_0\rho_1(1-\rho_2), \ldots,
$$
then it can be seen that, if $\rho_k > \frac 12$ for all but finitely many $k$s, then $F$ contains intervals. If, on the other hand,  $\rho_k < \frac 12$ for all $k \ge 0$, then $F$ is a Cantor set. If in addition to this, $\mathop{\lim}\limits_{k \to \infty} 2^k \rho_0\ldots \rho_{k-1} = 0$, then $F$ has zero Lebesgue measure and $\mu$ is singular. 

\fr A particular though interesting case is when $r_k = \rho^k, k \ge 0$, for some $\rho \in (0, 1)$. Then the corresponding set $F=F_\rho$ is the set of real numbers of type $\pm1 \pm \rho \pm \rho^2 \pm \ldots$. If $\rho < \frac 12$, then $F_\rho$ has zero Lebesgue measure and $\mu^{(\rho)}$ is singular; \ if $\rho> \frac 12$, then $F_\rho$ contains intervals. The convolution $\mu^{(\rho)}$ is equal to the invariant probability of the IFS with two contractions 
$$
\phi_1(x) = \rho x +1, \ \phi_2(x) = \rho x-1,
$$
taken with probabilities $1/2$, $1/2$. This is a conformal system with overlaps, and $F_\rho$ is equal to  the limit set $J_\rho$ of this IFS. The measure $\mu^{(\rho)}$ is the projection $\nu_{(1/2,1/2)}\circ\pi^{-1}$ of the probability $\nu_{(1/2, 1/2)}$ from $\{1, 2\}^\N$, through the canonical projection $\pi: \{1, 2\}^\N \to J_\rho$. In \cite{E} Erd\"os proved that when $1/\rho$ 
is a Pisot number (i.e a real algebraic integer greater than 1 so that all its conjugates are less than 1 in absolute value), then the measure $\mu^{(\rho)}$ is singular. In \cite{puaffine} it was shown that its Hausdorff dimension is strictly smaller than $1$. In the other direction, B. Solomyak showed in \cite{So} that $\mu^{(\rho)}$ is absolutely continuous for a.e $\rho \in [1/2, 1)$. 

\

Here we will give several ways to extend and randomize the idea of this construction, and will apply our results on pointwise dimensions of projection measures for random infinite IFS with overlaps:

\

\textbf{Random system 4.1.1}

A  type of random IFS can be obtained by fixing numbers $r_1, r_2 \in (0, 1)$, letting $\Lambda = \{1, 2\}^\Z$, $\theta: \Lambda \to \Lambda$ be the shift homeomorphism, and setting $E = \{1, 2\}$ so the alphabet is finite in this case. For arbitrary $\lambda=(\ldots, \lambda_{-1}, \lambda_0, \lambda_1, \ldots) \in \La$ and $e \in E$, consider then the affine contractions $\phi_e^\lambda$ in one real variable, defined by:
\begin{equation}\label{pm}
\phi_1^\lambda(x) = r_{\lambda_0} x + 1, \ \ \phi_2^\lambda(x) = r_{\lambda_0} x - 1
\end{equation} 

\fr Then, for arbitrary $\lambda = (\ldots, \lambda_{-1}, \lambda_0, \lambda_1, \ldots\} \in \{1, 2\}^\Z$,  the corresponding fractal limit set is $$J_\lambda := \pi_\lambda(E^\infty) = \{\phi_{\omega_1}^\lambda \circ \phi_{\omega_2}^{\theta(\lambda)} \circ \ldots, \ \omega = (\omega_1, \omega_2, \ldots)  \in E^\infty\},$$ which can actually be described as a set of type 
$$
\lt\{\pm 1 + \mathop{\sum}\limits_{i\ge 1} \mathop{\sum}\limits_{(j, k) \in Z_i}\pm \rho_1^k\rho_2^j\rt\},
$$ 
where for any pair of positive integers $(j, k) \in Z_i$ we have $j+k = i, \ i \ge 1$, and where the sets $Z_i$ are prescribed by the parameter $\lambda \in \{1, 2\}^\Z$, while the signs $\pm$ are arbitrary.

\fr We then consider the 1-sided shift space $E^\infty$, and a Bernoulli measure $\nu = \nu_Q$ on $E^\infty$ given by a probability vector $Q = (q_1, q_2)$. Let also a Bernoulli measure $m = m_P$ on $\Lambda$ associated to the probability vector $P = (p_1, p_2)$, and the probability $\mu = m \times \nu$ on $\La \times E^\infty$. The above random finite IFS is denoted by $\cS$.

\fr Next, by desintegrating $\mu$ into conditional measures $\mu_\lambda$,  and projecting $\mu_\lambda$ to the limit set $\J_\lambda$, we obtain the projection measure $\mu_\lambda \circ \pi_\lambda, \ \lambda \in \La$. In this case the finiteness condition of entropy from the statement of Theorem \ref{l3ie27} is clearly satisfied since $E$ is finite, so we obtain the exact dimensionality of the measures $\mu_\lambda \circ \pi_\lambda^{-1}$ on $J_\lambda$ for $m$-almost all $\lambda \in \La$. And from Corollary \ref{ptwdim} and Theorem \ref{ab}, we obtain 

\begin{cor}\label{4.1.1}
In the setting of 4.1.1, we obtain the following upper estimate for the pointwise (and Hausdorff, packing) dimensions of the projection measures, for $\mu$-almost all $(\lambda, \omega) \in  \Lambda \times E^\infty$:
$$d_{\mu_\lambda \circ \pi_\lambda^{-1}}(\pi_\lambda(\omega)) = \frac{h_\mu(\cS)}{\chi_\mu} \le \frac{h(m_P) + h(\nu_Q)}{- p_1 \log r_1 - p_2 \log r_2}= \frac{p_1 \log p_1 + p_2 \log p_2 + q_1 \log q_1 + q_2 \log q_2}{p_1 \log r_1 + p_2 \log r_2}$$
\end{cor}

\

Also, another possibility is to take  $\mu = m \times \nu$ on $\Lambda \times E^\infty$, where $m = m_P$ as before and $\nu$ is an equilibrium measure of a H\"older continuous potential on the 1-sided shift space $E^\infty$. 

\

 \textbf{Random system 4.1.2}

Consider now a fixed sequence $\bar \rho= (\rho_i)_{i \ge 1}$ of numbers in $(0, 1)$ which are smaller than some fixed $\rho \in (0, 1)$, and let the parameter space $\Lambda = \{1, 2, \ldots\}^\Z$ with the shift homeomorphism $\theta: \La \to \La$. Let also an infinite probability vector $P= (p_1, p_2, \ldots)$, and the $\theta$-invariant Bernoulli measure $m_P$ on $\La$ satisfying $m_P([i]) = p_i, \ i \ge 1$, where $[i]:= \{\omega = (\ldots, \omega_{-1}, \omega_0, \omega_1, \ldots), \omega_0 = i\}, \ i \ge 1$, and $h(\nu_P) < \infty$.
Let us take then the set $E:= \{1, 2, \ldots\}$ and a $(\theta\times\sigma)$-invariant probability measure $\mu$ on $\Lambda \times E^\infty$, having its marginal on $\Lambda$ equal to $m_P$. For example we can take $\mu = m_P \times \nu_Q$, where $Q= (q_1, q_2, \ldots)$ is a probability vector, and where $\nu_Q([j]) = q_j, j \ge 1$ is a $\sigma$-invariant Bernoulli probability on $E^\infty$; we assume in addition that the entropy of $\nu_Q$ is finite, i.e that $$-\sum_{j\ge 1} q_j \log q_j < \infty$$

\fr We now define infinitely many contractions $\phi_e^\lambda$ on a fixed large enough compact interval $X$, for arbitrary $e \in E, \lambda = (\ldots, \lambda_{-1}, \lambda_0, \lambda_1, \ldots) \in \La$, \ $\lambda_i \in \{1, 2\ldots\}, i \in \Z$,  by:
$$\phi_{n}^\lambda(x) = \rho_{\lambda_n} \cdot x +(-1)^{\lambda_0},  \ n \ge 1$$ 
It is clear that $\phi_e^\lambda$ are conformal contractions and they satisfy Bounded Distortion Property. We construct thus a random infinite IFS denoted by $\cS(\bar \rho)$, which has overlaps.  

For every $\lambda \in \La$, we construct then the fractal limit set $J_\lambda:= \pi_\lambda(E^\infty)$, which may be non-compact. The fractal $J_\lambda$ is the set of points given as $\phi^\lambda_{\omega_1} \circ \phi^{\theta(\lambda)}_{\omega_2} \circ \ldots$, for all $\omega \in E^\infty$. The main difference from the previous  example 4.1.1 is that now, the plus and minus signs in the series giving the points of $J_\lambda$ are \textit{not arbitrary}, instead they are determined by $\lambda = (\ldots, \lambda_{-1}, \lambda_0, \lambda_1, \ldots)  \in \La$. The randomness in the series comes now from the various possibilities to choose the sequences $\omega =(\omega_0, \omega_1, \ldots)  \in E^\infty$. Thus,
$$J_\lambda = \{(-1)^{\lambda_0} + (-1)^{\lambda_1}\rho_{\lambda_{\omega_1}} + (-1)^{\lambda_2}\rho_{\lambda_{\omega_1}} \rho_{\lambda_{\omega_2}} + \ldots + x \cdot \rho_{\lambda_{\omega_1}} \rho_{\lambda_{\omega_2}} \ldots, \  \ \omega_i \in \N^*, i \ge 0\}$$

\fr Given the $(\theta\times \sigma)$-invariant probability measure $\mu = m_P \times \nu_Q$, we see from Remark \ref{finiteent} that the condition $H_\mu(\pi_{E^\infty}^{-1}(\xi) | \pi_\Lambda^{-1}(\epsilon_\Lambda)) < \infty$ is satisfied. 
For arbitrary $\lambda \in \Lambda$, we now take the projection measure $\mu_\lambda\circ \pi_\lambda^{-1}$ on $J_\lambda$.
\ Therefore, from Theorem \ref{l3ie27} and Corollary \ref{ptwdim} we obtain that for $m_P$-almost all $\lambda \in \La$, the measure $\mu_\lambda\circ\pi_\lambda^{-1}$ is exact dimensional and its pointwise dimension has a common value equal  to $h_\mu(\cS(\bar \rho))/\chi_\mu$, where in our case the Lyapunov exponent of $\mu$ with respect to the random infinite  system $\cS(\bar \rho)$ is equal to:
$$
\begin{aligned}
\chi_\mu 
&= -\int_{\La \times E^\infty} \log \rho_{\lambda_{\omega_1}} d\mu(\lambda, \omega) = -\sum_{i \ge 1} q_i \int \log \rho_{\lambda_i} dm_P(\lambda)\\
& = -\sum_{i, j \ge 1} p_j q_i \log \rho_j = -\sum_{j \ge 1} p_j \log \rho_j.
\end{aligned}
$$ 
Moreover we have from Theorem \ref{ab} that the random projectional entropy of $\mu$ satisfies  
$$
h_\mu(\cS(\bar \rho)) \le h(\mu) = h(\mu_P) + h(\nu_Q) = -\sum_{i \ge 1} p_i\log p_i - \sum_{j \ge 1} q_j \log q_j. 
$$ 
 
\begin{cor}\label{4.1.2}
In the setting of 4.1.2, we obtain a concrete upper estimate for the pointwise (and Hausdorff, packing) dimension of $\mu_\lambda\circ \pi_\lambda^{-1}$, for $\mu$-almost all $(\lambda, \omega)\in \Lambda \times E^\infty$:
$$
d_{\mu_\lambda\circ \pi_\lambda^{-1}}(\pi_\lambda(\omega)) \le \frac{ \mathop{\sum}\limits_{i \ge 1} p_i\log p_i +\mathop{\sum}\limits_{j \ge 1} q_j \log q_j}{\mathop{\sum}\limits_{j \ge 1} p_j \log \rho_j}
$$
\end{cor}

\

\textbf{Random system 4.1.3}

Let us fix a sequence $\bar \rho = (\rho_0, \rho_1, \rho_2, \ldots)$ in $(0, 1)$, and $\Lambda= [1-\vp, 1+\vp]$ for some small $\vp >0$, together with a homeomorphism $\theta: \La \to \La$ which preserves an absolutely continuous probability $m$ on $\Lambda$. Let us take also the set $E = \{1, 2, \ldots\}$ and the $\sigma$-invariant Bernoulli measure $\nu$ on $E^\infty$ given by $\nu([i]) = \nu_i, i \ge 1$, where $(\nu_1, \nu_2, \ldots)$ is a probability vector. We assume also that $h(\nu) = -\mathop{\sum}\limits_{i\ge 1} \nu_i \log \nu_i  < \infty$. For arbitrary $e \in E$ and $\lambda \in \La$, we now define the sequence of parametrized contractions: 
$$
\phi_{2n+1}^\lambda(x) = \lambda \rho_n x +1, \ \ \phi_{2n+2}^\lambda(x) = \lambda\rho_n x -1, \ n \ge 0.
$$
By considering also the $(\theta\times\sigma)$-invariant probability $\mu = m \times \nu$ we obtain the random infinite IFS with overlaps $\cS(\bar \rho)$.

\fr  The corresponding limit set $J_\lambda:= \pi_\lambda(E^\infty)$ can be thought of as the set determined, for $\lambda \in \Lambda$,  in the following way:  $$J_\lambda= \{\pm 1 \pm \lambda \rho_{i_1} \pm \lambda^2 \rho_{i_1}\rho_{i_2} \pm \ldots,  \text{for all sequences of positive integers} \ \omega = (i_1, i_2, \ldots) \}$$
 The projection $(\pi_\lambda)_*\mu_\lambda = \mu_\lambda \circ \pi_\lambda^{-1}$ of the measure $\mu_\lambda$, is a probability measure on $J_\lambda$. We see that both  (\ref{bdp}) and the entropy condition $H_\mu(\pi_{E^\infty}^{-1} \xi \ \pi_\Lambda^{-1} \epsilon_\La) < \infty$, are satisfied in this case. 

\fr Hence, we can apply Theorem 
\ref{l3ie27} and Corollary \ref{ptwdim}, to obtain that for $m$-almost all parameters $\lambda \in [1-\vp, 1+ \vp]$, the projection measure $\mu_\lambda\circ \pi_\lambda^{-1}$ is exact dimensional, and that its Hausdorff dimension has a common value, which is  equal to $$HD(\mu_\lambda\circ \pi_\lambda^{-1}) = \frac{h_\mu(\cS(\bar \rho))}{\chi_\mu},$$
where the Lyapunov exponent of $\mu$ with respect to $\cS(\bar \rho)$ is given by:
$$\chi_\mu = -\int_{\La \times E^\infty} \log( \lambda\rho_{[\frac{\omega_1-1}{2}]}) \ d\mu(\lambda, \omega) = -\int_\La \log \lambda \ dm(\lambda) - \sum_{i \ge 0} (\nu_{2i+1} + \nu_{2i+2}) \log\rho_i$$
\fr From Theorem \ref{ab} we obtain an upper estimate for the random projectional entropy,  $h_\mu(\cS) \le h(m) -\sum_i \nu_i \log \nu_i$,  and an upper estimate for the pointwise dimension and Hausdorff dimension of $\mu_\lambda\circ \pi_\lambda^{-1}$. 

\begin{cor}\label{4.1.3}
 In the setting of 4.1.3, we obtain that for $\mu$-almost every $(\lambda, \omega) \in [1-\vp, 1+\vp] \times E^\infty$, 
$$d_{\mu_\lambda\circ \pi_\lambda^{-1}}(\pi_\lambda(\omega)) = HD(\mu_\lambda\circ \pi_\lambda^{-1}) \le \frac{h(m) -\sum_{i\ge 1} \nu_i \log \nu_i}{-\int_\La \log \lambda \ dm(\lambda) - \sum_{i \ge 0} (\nu_{2i+1} + \nu_{2i+2}) \log \rho_i}$$
\end{cor}

\fr If all the contraction factors $\rho_i$ are equal to some fixed $\rho$, then $J_\lambda$ is a perturbation of the set from the beginning of 4.1.

\

\textbf{4.2. Random continued fractions.}

\

 By the continued fraction $[a_1, a_2, \ldots]$ with digits $a_1, a_2, \ldots$, we understand the ratio $$\frac{1}{a_1+ \frac{1}{a_2 + \frac{1}{a_3+\ldots}}}$$
R. Lyons studied in \cite{L} random continued fractions $[1, X_1, 1, X_2, \ldots]$, where the random variables $X_i, i \ge 1$ are i.i.d and take the values $0, \alpha$ each with probability $1/2$, and where $\alpha$ is a fixed number in $(0, \infty)$.
Let $\nu_\alpha$ be the distribution of this random continued fraction. 
In fact the measure $\nu_\alpha$ is the invariant measure of the iterated function system $\tilde{\mathcal S}_\alpha = \{\phi_1^\alpha, \phi_2\} = \{\frac{x+\alpha}{x+\alpha+1}, \frac{x}{x+1}\}$, where the two generator maps are applied with equal probabilities.

 If $P_\alpha$ is the fixed point of $\phi^\alpha_1$, then $P_\alpha = \frac{-\alpha + \sqrt{\alpha^2+4\alpha}}{2}$, and it can be seen that $\phi_1^\alpha(0) > \phi_2(P_\alpha)$ if and only if $\alpha > \frac 12$. Thus for $\alpha > \frac 12$, the support of $\nu_\alpha$ is a Cantor set contained in $\phi_1^\alpha([0, P_\alpha)] \cup \phi_2([0, P_\alpha])$. If $\alpha \in (0, \frac 12)$, then there are strict overlaps and the  limit set of $\tilde{\mathcal S}_\alpha$ is the interval $[0, \frac{-\alpha + \sqrt{\alpha^2 + 4\alpha}}{2}]$.

Lyons showed in \cite{L} that the measure $\nu_\alpha$ is singular for all $\alpha \in (\alpha_c, \frac 12]$, where $\alpha_c \in (0.2688, 0.2689)$. 
Later, by employing a transversality condition, Simon, Solomyak and Urba\'nski \cite{SSU} made progress and showed  that $\nu_\alpha$ is absolutely continuous for Lebesgue-almost all $\alpha \in (0.215, \alpha_c)$. They left it open whether $\nu_\alpha$ is absolutely continuous or singular for small values of $\alpha$,  as transversality may fail in that case.  

In our paper, we do not use transversality, but view the original system $\tilde{\mathcal S}_\alpha$ as a random parabolic IFS with overlaps.
The system $\tilde{\mathcal S}_\alpha$ is not hyperbolic (the map $\phi_2$ is not contracting everywhere), hence we cannot apply directly our results above. However we can associate to it an \textit{infinite random IFS with overlaps} containing only contractions, using the jump transformation (\cite{Sch}).  
In this way, infinite random IFS of contractions with overlaps, will appear naturally in this situation. 
 
Our goal is to show that the invariant measure $\nu_\alpha$ of the IFS $\tilde{\mathcal S}_\alpha$ is \textit{exact dimensional}, and to give \textit{estimates} on its pointwise dimension.  Once we have exact dimensionality, it means that all fractal invariants of the measure (pointwise dimension, Hausdorff dimension, packing dimension) are the same.

\begin{thm}\label{lyons}
Consider the IFS $\tilde{\mathcal S}_\lambda = \{ \phi^\lambda_1, \phi_2\} = \{\frac{x+\lambda}{x+\lambda+1}, \frac{x}{x+1}\}$, and let the invariant measure $\nu_\lambda$ obtained by applying the maps of $\tilde{\mathcal S}_\lambda$ each with probability $\frac 12$. Denote also by $\mathcal S$ the associated infinite IFS obtained above by the jump transformation. 

\ a)  Then, for Lebesgue-almost all $\lambda \in [0, 1]$, the measure $\nu_\lambda$ is exact dimensional, and its pointwise (and Hausdorff, packing) dimension is equal to $\frac{h_\mu(\mathcal S)}{\chi_\mu}$, where $\mu = m \times \tilde \nu$ and $\tilde \nu$ is the Bernoulli measure on $\{0, 1, 2, \ldots\}^\infty$ associated to the vector $(\frac 12, \frac{1}{2^2}, \ldots)$. 

\ b) For Lebesgue-almost all $\lambda \in [\frac{-1+\sqrt 3}{2}, 0.5]$, the pointwise (and Hausdorff) dimension of $\nu_\lambda$ is larger than $\frac{1}{25}$. 

\end{thm}

\begin{proof}
In order to keep our notation for random systems, replace $\alpha$ by $\lambda$. We have the maps: $$\phi^\lambda_1(x) = \frac{x+\lambda}{x+\lambda+1}, \  \phi_2(x) = \frac{x}{x+1}$$
Consider $X = [0, 1], \Lambda = [0, 1]$, and the transformation  $\theta(\lambda) = \lambda$ on $\Lambda$, which invariates the Lebesgue measure $m$. 
We see that both $\phi^\lambda_1$ and $\phi_2$ are increasing, and that $\phi_1^\lambda(0) = \frac{\lambda}{\lambda+1}, \phi_1^\lambda(1) = \frac{\lambda+1}{\lambda+2}$, and $\phi_2(0) = 0, \phi_2^\lambda(1) = \frac 12$. Also the map $\phi^\lambda_1$ is a contraction for $\lambda >0$, but $\phi_2$ is not contracting. 
According to the jump transformation (\cite{Sch}), we can associate to our parabolic IFS $\tilde{\mathcal S}_\lambda$, a hyperbolic IFS $\mathcal S_\lambda = \{\psi^\lambda_n, n \ge 0\}$, formed by the transformations of $X$, 
$$\psi^\lambda_n = \phi_2^n \circ \phi_1^\lambda, \ n \ge 0$$
From the definition, it follows that the maps $\psi^\lambda_n$ are all contractions, and by induction on $n$, 
$$\psi^\lambda_n(x) = \frac{x+\lambda}{(n+1)(x+\lambda) +1}$$
We denote by $\mathcal{S}$ the random infinite IFS with overlaps, obtained above.
For $\lambda >0$, the measure $\nu_\lambda$ is the projection onto the limit set $J_\lambda$ of the Bernoulli measure on $\{1, 2\}^{\mathbb N}$, associated  to  $(\frac 12, \frac 12)$. If $\phi_*\mu$ denotes in general the push forward of a measure $\mu$ through a map $\phi$, then  $\nu_\lambda$ is the unique probability measure $\nu$ satisfying the condition: 
\begin{equation}\label{12}
\nu = \frac 12 \phi_{1*}^{\lambda}\nu + \frac 12 \phi_{2*}\nu
\end{equation}

This means that $\nu_\lambda$ satisfies also the equality:
$$\nu_\lambda = \frac 12 \phi_{1*}^\lambda \nu_\lambda + \frac{1}{2^2} (\phi_2\phi_1^\lambda)_*\nu_\lambda + \frac{1}{2^3} (\phi_2^2\phi_1^\lambda)_*\nu_\lambda + \ldots = \frac 12 \psi^\lambda_{0*}\nu_\lambda + \frac{1}{2^2}\psi^\lambda_{1*}\nu_\lambda + \frac{1}{2^3}\psi^\lambda_{2*} \nu_\lambda+ \ldots $$
So $\nu_\lambda$ satisfies the corresponding identity for 
$\mathcal S$ and for the probability vector $(\frac 12, \frac{1}{2^2}, \ldots)$, hence $\nu_\lambda$ it is the projection of the Bernoulli measure $\tilde \nu$ on $E^\infty$ associated to  $(\frac{1}{2}, \frac{1}{2^2}, \ldots)$, where $E = \mathbb N$. 

If we take now the measure $\mu = m \times \tilde \nu$ on $\Lambda \times E^\infty$, we have that $h(\mu) = h(\tilde \nu) = \mathop{\sum}\limits_{n \ge 1} \frac{\log 2^n}{2^n} = \log 2 \mathop{\sum}\limits_{n \ge 1} \frac{n}{2^n} = 2 \log 2 < \infty$. Then, Remark \ref{finiteent} applies, and we obtain from Theorem 
 \ref{l3ie27} that for $\mu$-almost all $(\lambda, \omega) \in \Lambda \times E^\infty$, the projection $\nu_\lambda$ is exact dimensional at the point $\pi_\lambda(\omega)$.
The Lyapunov exponent of the measure $\mu$ is defined by:
$$\chi_\mu = - \int_{\Lambda \times E^\infty} \log \|(\psi_{\omega_1}^\lambda)' (\pi_{\theta(\lambda)}(\sigma\omega)\| \ d\mu(\lambda, \omega)$$
From the above formula for $\psi^\lambda_n$, \  $(\psi^\lambda_n)'(x) = \frac{1}{[(n+1)(x+\lambda)+1]^2}$; thus, since $x \in [0, 1]$,
$$
\frac{1}{[\lambda(n+1) +n+2]^2}\le |((\psi^\lambda_n)'(x)| \le \frac{1}{[\lambda(n+1)]^2}
$$
Hence we obtain:
$$
\chi_\mu \le \mathop{\sum}\limits_{n \ge 1} \int_0^1 \frac{1}{2^n}\log[\lambda(n+1)+n+2] \ d\lambda \le \mathop{\sum}\limits_{n \ge 1} \frac{\log [2(n+2)]}{2^n}  = 
\log 2 + \mathop{\sum}\limits_{n \ge 1} \frac{\log (n+2)}{2^n} <\infty
$$
Therefore from Theorem  \ref{l3ie27} we obtain the \textit{lower estimate for the pointwise dimension},
\begin{equation}\label{le}
d_{\mu_\lambda\circ \pi_\lambda^{-1}}(\pi_\lambda(\omega)) = HD(\mu_\lambda\circ \pi_\lambda^{-1}) = \frac{h_\mu(\mathcal S)}{\chi_\mu} \ge \frac{h_\mu(\mathcal S)}{\log 2 + \mathop{\sum}\limits_{n \ge 1} \frac{\log (n+2)}{2^n}}
\end{equation}

\

We want now to estimate the random projectional entropy $h_\mu(\mathcal S)$ for certain values of $\lambda$, which will give estimates also the pointwise dimension. 
Firstly, we know that $\psi^\lambda_n(0) = \frac{\lambda}{\lambda(n+1) +1} = \frac{1}{n+1+\frac{1}{\lambda}}, \psi^\lambda_n(1) = \frac{\lambda+1}{(n+1)(\lambda+1) + 1} = \frac{1}{n+1+\frac{1}{\lambda+1}}$, and that clearly $(\psi^\lambda_n(0))_n$ and $(\psi^\lambda_n(1))_n$ are strictly decreasing sequences in $n$, as $\lambda >0$. 

We want to see what is the maximum number of intervals $I_j^\lambda:= \psi^\lambda_j([0, 1])$ that any given interval $I_n^\lambda$ intersects. Assuming that $I_{n+k}^\lambda$ intersects $I_n^\lambda$, it follows that $\frac{1}{n+k+1+\frac{1}{\lambda+1}} > \frac{1}{n+1+\frac{1}{\lambda}}$, thus $\lambda(\lambda+1) < \frac 1k$. Looking next at intervals of type $I^\lambda_{n-k'}$, we have an overlap between $I^\lambda_n$ and $I^\lambda_{n-k'}$ iff $\frac{1}{n-k'+1+\frac 1\lambda} < \frac{1}{n+1+\frac{1}{\lambda+1}}$, which means again that $\lambda(\lambda+1) < \frac{1}{k'}$. By combining, we obtain that $I^\lambda_{n}$ intersects $k + k'$ intervals $I^\lambda_j, j \ne n$, \ if $\lambda(\lambda+1) < \text{max}\{\frac{1}{k}, \frac{1}{k'}\}$. 
In particular, if $\lambda > \frac{-1+\sqrt{3}}{2}$,  then each interval $I_n^\lambda$ intersects strictly less than 4 other intervals $I^\lambda_j$. 

Let us take now the parameter space $\Lambda$ to be the interval $(\frac{-1+\sqrt{3}}{2}, \frac 12)$, with the normalized Lebesgue measure $m$. The transformation $\theta$ is again the identity on $\Lambda$, $E = \mathbb N$, and the measure $\nu$ on $E^\infty$ is the Bernoulli measure associated to the probability vector $(\frac 12, \frac{1}{2^2}, \ldots)$. Recall that the entropy $h(\nu)$ is equal to $2 \log2$. 

Since for $\lambda > \frac{-1+\sqrt 3}{2}$, each image $\psi^\lambda_n([0, 1])$ intersects at most 3 other images $\psi^\lambda_j([0, 1])$, it follows from Theorem \ref{ab} that the random projectional entropy $h_\mu(\mathcal S)$  can be estimated as: $$h_\mu(\mathcal S) \ge h(\mu) - \log 3 = 2\log 2 - \log 3 = \log \frac 43$$

On the other hand, from the definition of the random Lyapunov exponent, we have
\begin{align}
\chi_\mu &\le \mathop{\sum}\limits_{n \ge 1} \frac{1}{2^n( \frac 12 - \frac{-1+\sqrt 3}{2})} \int_{\frac{-1+\sqrt 3}{2}}^{\frac 12} \log[\lambda(n+1) + n+2]  d\lambda \le 10 \mathop{\sum}\limits_{n \ge 1} \frac{\log\frac 32 (n+2)]}{2^n}\\ &\le 10 (\log \frac 32 + \mathop{\sum}\limits_{n \ge 1} \frac{\log(n+2)}{2^n})
\end{align}
Thus the pointwise dimension is larger than $\frac{1}{25}$ for Leb-a.a $\lambda \in [\frac{-1+\sqrt 3}{2}, \frac 12]$. The value $\frac {1}{25}$ is not optimal.
\end{proof}

Another possibility, is to take  $\Lambda = [0, 1]$,  $\theta: \Lambda \to \Lambda$ to be an expanding smooth bijective map, and $m$  its absolutely continuous invariant probability measure on $[0, 1]$. We can then form the random system $\mathcal{S}$ and the measure $\nu_\lambda$ as before. By applying our results, we obtain that $\nu_\lambda$ is exact dimensional, and we can estimate its Hausdorff (and packing) dimension.

We make the observation that our method can be applied also to other random continued fractions, whose digits take values in some fixed set.
 
\

\textbf{4.3. Random infinite IFS with bounded number of overlaps in the plane.} 

\

In Example 5.11 of \cite{MU}, we gave an example of a deterministic infinite IFS defined as follows:  \
let $X = \bar B(0, 1) \subset \R^2$ be the closed unit disk and for $n \ge 1$ take $C_n$ to be the circle centered at the origin and having radius $r_n \in (0, 1), \ r_n \mathop{\nearrow}\limits_{n \to \infty} 1$. 
For each $n\ge 1$ we cover the circle $C_n$ with closed disks $D_n(i), i\in K_n$, of the same
radius $r_n'$, where $K_n$ is a finite set and each disk $D_n(i)$ intersects only
two other disks of the form $D_n(j), j \in K_n$, and where none of the disks $D_n(i)$ intersects $C_k, k \ne n$.  Moreover, we assume that for any $m \ne n, m , n \ge 1$, the
families $\{D_m(i)\}_{i\in K_m}$ and $\{D_n(i)\}_{i \in K_n}$ consist of mutually disjoint disks.
Consider contraction  similarities
$\phi_{n, i}: X \to X,  i \in K_n, n \ge 0$ whose respective images of $X$ are the above disks
$D_n(i), i \in K_n, n \ge 0$. For this deterministic system, the boundary at infinity $\partial_\infty \mathcal{S}$ is contained in $\partial X$.
 
Assume now in addition, that there exists $\vp>0$, such that for $m \ne n$,  any disk $(1+\vp)D_n(i), i \in K_n$ does not intersect any disk of type $(1+\vp)D_m(j), j \in K_m$ \ (where in general for $\beta >0$, $\beta D_n(i)$ denotes the disk of the same center as $D_n(i)$ and radius equal to $\beta r_n'$), and that any disk $(1+\vp)D_n(i)$ intersects only two other disks $(1+\vp)D_n(j), j \in K_n$.

\fr We take now $\Lambda = [1-\vp, 1+\vp]$ and $\theta: \La \to \La$ a homeomorphism which preserves an absolutely continuous probability measure $m$ on $[1-\vp, 1+\vp]$. Let the following countable alphabet $$E = \{(n, i), \ i \in K_n, n \ge 0\},$$ which will be our alphabet. Consider also a fixed probability vector $P = (\nu_e)_{e \in E}$, and the associated Bernoulli probability $\nu = \nu_P$ on $E^\infty$, and  let us assume that $h(\nu) < \infty$. 

\fr We now define the conformal contraction
$\phi_{(n, i)}^\lambda(x)$, as being a similarity  with image $\phi_{(n, i)}^\lambda(X)$ equal to $\lambda D_n(i)$, for $i \in K_n, n \ge 0$ and $\lambda \in \La$; its contraction factor is equal to $\lambda r_n', n \ge 0$.

\fr Consider now the probability $\mu = m \times \nu$ defined on $\Lambda \times E^\infty$. We have constructed thus a random conformal infinite IFS with overlaps, denoted by $\cS$; and, from Remark \ref{finiteent} and since $h(\nu) < \infty$, we obtain also the finite entropy condition $H_\mu(\pi_{E^\infty}^{-1}\xi | \pi_\La^{-1}\epsilon_\La) < \infty$.

\fr The conditions in Theorem \ref{l3ie27} and Corollary \ref{ptwdim} are satisfied, and thus for Lebesgue-a.e $\lambda \in \Lambda$ and $\nu$-a.e $\omega \in E^\infty$, the projection measure $(\pi_\lambda)_*\mu_\lambda = \mu_\lambda \circ \pi_\lambda^{-1}$ on the non-compact limit set $J_\lambda:= \pi_\lambda(E^\infty)$ is exact dimensional, and its pointwise dimension is given by: $$d_{\mu_\lambda\circ \pi_\lambda^{-1}}(\pi_\lambda(\omega)) = \frac{h_\mu(\cS)}{\chi_\mu}, $$ where the Lyapunov exponent of $\mu$ with respect to the random system  $\cS$ is equal to: $$\chi_\mu = -\log \lambda - \sum_{e=(n, i) \in E}\nu_e \log r_n'>0$$

\fr From the construction of the disks $\lambda D_n(i), i \in K_n, n \ge 0, \lambda \in \La$ above,  we notice that the condition in Theorem \ref{ab}, part b) is satisfied with $k = 2$. Hence we can obtain a \textit{lower estimate} for the random projectional entropy of $\mu$, namely
$$h_\mu(\cS) \ge h(\mu) - \log 2 = h(m) - \sum_{e \in E} \nu_e \log \nu_e - \log 2$$
 Hence, by combining the last two displayed formulas and using Theorem \ref{ab}, we obtain:

\begin{cor}\label{4.3}
In the setting of 4.3, for $\mu$-almost every pair $(\lambda, \omega) \in \Lambda \times E^\infty$, the pointwise dimension of $\mu_\lambda\circ \pi_\lambda^{-1}$ satisfies the  following estimates:
\

$$\frac{h(m) - \mathop{\sum}\limits_{e \in E} \nu_e\log \nu_e - \log 2}{-\log \lambda - \mathop{\sum}\limits_{e=(n, i) \in E}\nu_e \log r_n'} \ \le  d_{\mu_\lambda\circ \pi_\lambda^{-1}}(\pi_\lambda(\omega)) \ \le \frac{h(m) - \mathop{\sum}\limits_{e \in E} \nu_e\log \nu_e}{-\log \lambda - \mathop{\sum}\limits_{e=(n, i) \in E}\nu_e \log r_n'}$$
\end{cor}

\

\

\end{document}